\documentclass[11pt]{amsart}
\usepackage{hyperref}
\newtheorem{theorem}{Theorem}[section]
\newtheorem{lemma}[theorem]{Lemma}
\newtheorem{corollary}[theorem]{Corollary}
\theoremstyle{definition}

\theoremstyle{remark}
\newtheorem{remark}[theorem]{Remark}
\newcommand{\x}{\bar{x}}
\newcommand{\y}{\bar{y}}

\newcommand{\s}{\tilde{v}}
\newcommand{\e}{\varepsilon}
\newcommand{\R}{\mathbb{R}}
\newcommand{\0}{\textbf{0}}
\numberwithin{equation}{section}
\DeclareMathOperator{\cexp}{c-exp}

\begin{document}

\title{\textbf{Boundary $\varepsilon$-regularity in optimal transportation}}

\author{Shibing Chen, Alessio Figalli}
\address{Department of mathematics, Zhejiang University of Technology, Hangzhou 310023, China}
\address{Mathematical Science Institute,
The Australian National University,
Canberra  ACT  2601}
\email{chenshibing1982@hotmail.com}

\address{Department of Mathematics, The University of Texas at Austin, 1 University Station C1200, Austin
TX 78712, USA}
\email{figalli@math.utexas.edu}
\thanks{}


\date{February 10}

\dedicatory{}


 \begin{abstract}
 We develop an $\e$-regularity theory at the boundary for a general class of Monge-Amp\`ere type
 equations arising in optimal transportation. As a corollary we deduce that optimal transport maps
between H\"older densities supported on $C^2$ uniformly convex domains are $C^{1,\alpha}$ up to the boundary,  provided that the cost function is a sufficient small perturbation of the quadratic cost $-x\cdot y$.
 \end{abstract}

 \maketitle

 \section{Introduction}
 Let $f$ and $g$ be two probability densities supported respectively on
 two bounded domains $X$ and $Y$ in $\mathbb{R}^n$.
 Let
 $c: X\times Y \rightarrow \mathbb{R}$ be a cost function.
 The optimal transport
problem is about finding a  map $T:X \to Y$ among all transport maps minimizing the transportation cost
$$\int_{X} c(x, T(x)) f(x)dx,$$
 where the term ``transport map'' means $T_\sharp f=g$.  Existence and uniqueness of optimal transport maps under mild conditions are now 
 well understood, see for instance \cite{C1} and \cite{GM}. 
 The regularity theory of optimal transport map with quadratic cost $c(x,y)=-x\cdot y$ \footnote{Usually
 one refers to $c(x,y)=|x-y|^2/2$ as ``quadratic cost''. However it is well-known that the costs $|x-y|^2/2$ and $-x\cdot y$ are completely equivalent (see for instance \cite[Section 3.1]{DFreview}), thus we will use the term ``quadratic cost'' to refer to the latter.} has been developed by Caffarelli  \cite{C90,C902,C91,C92,C96} (see also \cite{U}).
 Since this paper is concerned with the regularity theory up to the boundary, we state here Caffarelli's global regularity result:
 
 \begin{theorem} \cite{C96} \label{t0}
 Suppose that $0<f\in C^{\alpha}(\overline X)$ and $0<g\in C^{\alpha}(\overline Y)$, where
 $X$ and $Y$ are uniformly convex bounded domains of class $C^2$. Then the optimal transport map $T$ associated to the cost $c(x,y)=-x\cdot y$ is of class $C^{1,\alpha}$ up to 
 the boundary of $X.$
 \end{theorem}
 
 For general costs, Ma, Trudinger, and Wang \cite{MTW} found the so called ``MTW condition'', which guarantees the smoothness of the optimal transport map
 provided the densities are smooth and the domains satisfy some suitable convexity conditions. Their condition reads as follows:
 \begin{equation}\label{mtw}
\underset{i,j,k,l,p,q,r,s}{ \sum}c^{p,q}(c_{ij,p}c_{q,rs}-c_{ij,rs})c^{r,k}c^{s,l}\xi_i\xi_j\eta_k\eta_l\geq 0 \qquad \text{in }X\times Y
 \end{equation}
 for all $\xi,$ $\eta\in \mathbb{R}^n$ satisfying $\xi\perp\eta$, where
lower indices before (resp. after) the comma indicates derivatives with respect to $x$ (resp. $y$)
(so for instance  $c_{i,j}=\frac{\partial^2 c}{\partial x_i\partial y_j}$), $(c^{i,j})$ is the inverse 
 of $(c_{i,j})$, and all derivatives are evaluated at $(x,y) \in X\times Y$.
  Later, Loeper \cite{L1} showed that MTW condition is actually a 
 necessary condition for the optimal transport map to be smooth for any positive smooth densities. After the breakthroughs  in \cite{MTW,L1}, 
 many works have been devoted to the regularity theory of optimal transport map under MTW condition, to cite some see \cite{FL,Liu,TW1,TW2,FR,Loe2,LV,LTW,FRV-surfaces,KMC1,FRV-nec,FRV,FKM,FKM-spheres}.

 By now, regularity of optimal transport maps under the MTW condition is well understood. 
 However, several interesting costs do not satisfy this condition, for instance $c(x,y)=\frac{1}{p}|x-y|^p$ does not satisfy MTW condition
 when $p\in (1,2)\cup(2,\infty)$, and actually the class of costs satisfying the MTW condition is very restricted.
 
 Recently,  De Philippis and Figalli \cite{DF} obtained a partial regularity result for optimal transport problem with general cost
without assuming neither the MTW condition nor any convexity on the domains.
The key result in \cite{DF} consists in the interior versions of
 our Theorems \ref{t1} and \ref{t2}. 
 Roughly speaking they prove that, given any $0<\beta<1$, if there are interior points $x_0 \in X,y_0 \in Y$ such that the cost function is sufficiently close $-x\cdot y$ in $C^2$ near $(x_0,y_0)$, the densities are 
sufficiently close to $1$ in $C^0$ (resp. the densities are $C^\alpha$)  near $x_0$ and $y_0$ respectively, and the potential function $u$ is sufficiently close to $\frac12 |x|^2$ in $C^0$ near $x_0$,
  then $u$ is $C^{1,\beta}$ (resp. $C^{2,\alpha}$) in a neighbourhood of $x_0$.
  Exploiting that   semiconvex functions are twice differentiable almost everywhere,
  De Philippis and Figalli used this result to show that optimal maps are always smooth outside a closed set of measure zero (see the proof of \cite[Theorem 1.3]{DF} for more details). 
  
  In this paper we prove the analogue of De Philippis and Figalli's result when $x_0$ is on the boundary
  of $X$. As an application  we show that 
optimal transport maps
between H\"older densities supported on $C^{2,\alpha}$ uniformly convex domains are  $C^{1,\alpha}$ in the interior and $C^{1,\alpha'}$ up to the boundary for some $\alpha' \in (0,\alpha)$,  provided that the cost function is a sufficient small perturbation of the quadratic cost $-x\cdot y$.
This allows us to improve a recent result of Caffarelli, Gonz\'{a}les, and Nguyen \cite{CGN}, where they prove that the optimal transport map is of class $C^{1,\alpha}$ {\em strictly} in the interior of $X$ (more precisely, as the size of the perturbation on the cost goes to zero, the transport map is $C^{1,\alpha}$ in larger and larger domains which invade $X$, see \cite[Theorem 1.1]{CGN}). 

We note that, in our case, to obtain an almost everywhere regularity of transport maps on the boundary (as done in \cite{DF} for the interior case) we should prove that at almost every point on the boundary (with respect to the $(n-1)$-dimensional Hausdorff measure) our assumptions are satisfied. This seems to be a very delicate issue and it will be investigate in future works.\\

The proof of our $\e$-regularity result follows the lines of \cite{DF}, but 
compared to the interior case the boundary regularity presents many new additional difficulties,
and several new ideas
have to be introduced to overcome them.

Indeed, first of all, notice that one of the key steps in the proof of the interior regularity result of De Philippis and Figalli \cite{DF} is to construct a smooth approximating 
solution to the original problem, which is based on solving an optimal transportation problem with cost $-x\cdot y$ and constant densities. In their case, the condition that
the potential function $u$ is close to $\frac12|x|^2$ ensures that the approximating solution is smooth.
 But when we are around a point on the boundary of the domain one cannot
expect such approximating solution to be smooth. To handle this we find a new approximation
argument by using a suitable symmetrization trick, where we first slightly modify the domain and
then we reflect it around a boundary point with respect to the tangent hyperplane.

Second, for proving $C^{2,\alpha}$ regularity, the comparison principle plays very important role. Recall that
the proof of $C^{2,\alpha}$ regularity of solutions to Monge-Amp\`ere type equation is usually based on an iteration argument. 
In the interior case, at each iteration, the potential function $u$ solves  a Monge-Amp\`ere type equation in a sub-level
set of $u$ with Dirichlet boundary data. Then, one can construct an approximating solution $v$ which solves the standard
Monge-Amp\`ere equation with constant right hand side in a small convex neighborhood of the the sub-level set of $u$. The
comparison principle is used to compare the difference between $u$ and $v$.
For the interior case, the comparison principle can be proved more or less in
a standard way. However, around a point on the boundary of the domain, the Dirichlet data of the solutions
(both the original solution and the approximating solution) are not under control. Luckily, in our case, we can prove that
near $x_0$ the optimal map sends the boundary of source onto the boundary of the target, and this property allows us to show that 
the Neumann data
of the original solution and the approximating solution are very close. Hence, we are able to use comparison principle for Monge-Amp\`ere equation with 
mixed boundary data.\\

The rest of the paper is organized as follows.  In section 2 we introduce some notation and preliminaries, and state the main results.
Section 3 is devoted to the proof of Corollaries \ref{c1} and \ref{c2}. Section 4 contains the proof of Theorem \ref{t1}, and in the last section we prove Theorem~\ref{t2}.

 \smallskip
\noindent {\bf Acknowledgement.} The second author has been partially supported by NSF Grant DMS-1262411.
This material is based upon work supported by the National Science Foundation under Grant No. 0932078 000, while the authors were in residence 
at the Mathematical Sciences Research Institute in Berkeley, California, during the fall semester of 2013.

 \section{Preliminaries and main results}
 First, we introduce some conditions which should be satisfied by the cost. Let $X$  and $Y$ be
 two  bounded open subsets of $\mathbb{R}^n$. 
 \medskip
  
 (\textbf{C0}) The cost function is of class $C^3$ with $\|c\|_{C^3(X\times Y)}<\infty.$
 
 (\textbf{C1}) For any $x\in X$, the map $Y \ni y\mapsto D_xc(x,y)\in \mathbb{R}^n$ is injective.
 
 (\textbf{C2}) For any $y\in Y$, the map $X \ni x\mapsto D_yc(x,y)\in \mathbb{R}^n$ is injective.
   
 (\textbf{C3}) $\det(D_{xy}c)(x,y)\neq 0$ for all $(x,y)\in X\times Y.$
 
 \medskip
 
 A function $u: X\rightarrow \mathbb{R}$ is said {\it $c$-convex} if it can be written as
 \begin{equation} \label{121}
 u(x)=\underset{y\in Y}{\sup}\{-c(x,y)+\lambda_y\}
 \end{equation}
 for some family of constants $\{\lambda_y\}_{y \in Y}\subset \R$.
 Note that  $(\textbf{C0})$ and \eqref{121} imply that a $c$-convex function is semiconvex, namely, there exists some constant $K$ depending only
 on $\|c\|_{C^2(X\times Y)}$ such that $u+K|x|^2$ is convex. One immediate consequence of the semiconvexity is that $u$ is twice differentiable almost everywhere.
 
 It is well known that  $(\textbf{C0})$ and  $(\textbf{C1})$ ensure that there exists a unique optimal transport map, and there exists a $c$-convex function $u$ such that
 the optimal map is a.e. uniquely characterized in terms of $u$ 
 (and for this reason we denote it by $T_u$) by the relation
\begin{equation}
\label{eq:Tu}
 -D_xc(x, T_u (x))= \nabla u(x) \qquad \text{for a.e. }x.
\end{equation}
As explained for instance in \cite[Section 2]{DF} (see also \cite{DFreview}), the transport condition
$(T_u)_\#f=g$ implies that $u$ solves at almost every point the Monge-Amp\`ere type equation
\begin{equation}
\label{eq:MA T}
\det\Bigl(D^2u(x)+D_{xx}c\bigl(x,\cexp_x(\nabla u(x))\bigr) \Bigr)=\left|\det\left(D_{xy}c\bigl(x,\cexp_x(\nabla u(x))\bigr) \right) \right| \frac{f(x)}{g(\cexp_x(\nabla u(x)))},
\end{equation}
where $\cexp$ denotes the \textit{$c$-exponential map} defined as
\begin{equation}
\label{eq:cexp}
\text{for any $x\in X$, $y \in Y$, $p \in \R^n$},
\qquad
\cexp_x(p)=y \quad \Leftrightarrow \quad p=-D_xc(x,y).
\end{equation}
Notice that, with this notation, $T_u(x)= \cexp_x(\nabla u(x))$.

 For a $c$-convex function, analogous to the subdifferential for convex function, we can talk about its 
 $c$-subdifferential: If $u: X\rightarrow \mathbb{R}$ is a $c$-convex function as above, the {\it $c$-subdifferential} of $u$ at $x$ 
 is the (nonempty) set
 $$\partial_cu(x):=\bigl\{y\in \overline{Y}: u(z)\geq -c(z,y)+c(x,y)+u(x)\qquad \forall\,z\in X  \bigr\}.$$
 We also define \emph{Frechet subdifferential} of $u$ at $x$ as
 $$\partial^-u(x):=\{p\in\mathbb{R}^n: u(z)\geq u(x)+p\cdot(z-x)+o(|z-x|)\}.$$
 It is easy to check that $$y\in \partial_cu(x)\quad\Longrightarrow\quad -D_xc(x,y)\in \partial^-u(x).$$
 In the following, we use the notation 
 \begin{equation}
 \label{eq:Su}
 S_u(h):=\{x\in X:u(x)< h\}
 \end{equation}
 to denote the sub-level set of a function $u$. In fact, there is a
 more general concept of $c$-sub-level set of a $c$-convex function, namely, one can define 
 $$S(u, h, x_0, y_0):=\{x\in X: u(x)< c(x,y_0)-c(x_0, y_0)+u(x_0)+h\},$$ where $x_0\in X$ and $y_0\in \partial_cu(x_0)$.
 In this paper, we will always perform some transformation so that $x_0=\0$, $y_0=\0$, $u(\0)=0$, $c(x,\0)=0$, and the notation
 $S_u(h)$ will be enough for our purpose.

To state our main results we need to introduce some more notation.

We denote $x':=(x^1,\cdots,x^{n-1}) \in \R^{n-1},$ and $x=(x', x^n) \in \R^n$.
Given two domains $\mathcal{C}_1,\mathcal{C}_2\subset \R^n$, we are going to assume that there exist
two functions $P,Q:\R^{n-1}\to \R$ of class $C^2$
satisfying $P(\textbf{0})=Q(\textbf{0})=0$, $\nabla P(\textbf{0})=\nabla Q(\textbf{0})=\textbf{0}$,
and
\begin{equation}
\label{eq:PQ}
\{x^n>P(x')\}\cap B_{1/2}\subset\mathcal{C}_1\subset \{x^n>P(x')\}\cap B_{2},\quad
\{y^n>Q(y')\}\cap B_{1/2}\subset\mathcal{C}_2\subset \{y^n>Q(y')\}\cap B_{2}.
\end{equation}
Note that \eqref{eq:PQ} implies that 
\begin{equation}
\label{inclusion}
\mathcal{C}_1 \subset \{x^n \geq -2\|P\|_{C^2} \}, \quad
\mathcal{C}_2 \subset \{y^n \geq -2\|Q\|_{C^2} \}.
\end{equation}

In the following $K_1, K_2, K_3$, and $K$, are always used to denote some universal
 constants whose
 value may change depending on the context.
 In the next theorems we implicitly assume that $\mathcal C_1\subset X$ and $\mathcal C_2\subset Y$, so that the cost is defined and satisfies $(\textbf{C0})$-$(\textbf{C3})$ on  $\mathcal C_1\times  \mathcal C_2$.

 \begin{theorem}\label{t1}
Let $f, g$ be two densities supported in $\mathcal{C}_1$ and $\mathcal{C}_2$ respectively,
let $P$ and $Q$ be as in \eqref{eq:PQ}, and let $u: \mathcal{C}_1\rightarrow \mathbb{R}$ be
a $c$-convex function such that $\partial_{c}u(\mathcal{C}_1)\subset B_2$ and $(T_{u})_{\sharp}f=g$ (see \eqref{eq:Tu}).
Then, for every $\beta\in (0,1)$ there exist constants $\delta_0$,
$\eta_0>0$ such that the following holds: if
\begin{equation}\label{401}
\|P\|_{C^2}+ \|Q\|_{C^2}\leq \delta_0,
\end{equation}
\begin{equation}\label{402}
\|f-\textbf{1}\|_{L^\infty(\mathcal C_1)}+\|g-\textbf{1}\|_{L^\infty(\mathcal C_2)}\leq \delta_0,
\end{equation}
\begin{equation}\label{403}
\|c(x,y)+x\cdot y\|_{C^2(\mathcal{C}_1\times\mathcal{C}_2)}\leq \delta_0,
\end{equation}
and \begin{equation}\label{eq:uP}
\biggl\|u-\frac{1}{2}|x|^2\biggr\|_{L^\infty(\mathcal{C}_1)}\leq \eta_0,
\end{equation}
then $u\in C^{1,\beta}(\overline{\mathcal{C}_1\cap B_{\rho_0}})$ for sufficiently small $\rho_0$.
\end{theorem}

\begin{theorem}\label{t2}
Let $u, f, g, \eta_0, \delta_0$ be as in Theorem \ref{t1}, and assume in addition $f\in C^\alpha(B_{1/2}\cap \{x^n\geq P(x')\}),$  $g\in C^\alpha(B_{1/2}\cap \{y^n\geq Q(y')\}).$
 There exist small constants $\eta_1\leq \eta_0$ and $\delta_1\leq \delta_0$ such that, if $$\|P\|_{C^2}+ \|Q\|_{C^2}\leq \delta_1,$$
\begin{equation}\label{e41}
\|f-\textbf{1}\|_{L^\infty(\mathcal C_1)}+\|g-\textbf{1}\|_{L^\infty(\mathcal C_2)}\leq \delta_1,
\end{equation}
\begin{equation}\label{e42}
\|c(x,y)+x\cdot y\|_{C^2(\mathcal{C}_1\times\mathcal{C}_2)}\leq \delta_1,
\end{equation}
and
\begin{equation}\label{e43}
\biggl\|u-\frac{1}{2}|x|^2\biggr\|_{L^\infty(\mathcal{C}_1)}\leq \eta_1,
\end{equation}
then there exists $\rho_1>0$ small such that the following holds: for any point $z\in B_{\rho_1}\cap \{x_n= P(x')\}$ there exists a second order polynomial $p$ such that,
$$
|u(x)-p(x)|\leq C |x-z|^{2+\alpha'}\qquad \forall\,x\in B_{\rho_1}(z)\cap \{x_n\geq  P(x')\},
$$
where $C>0$ and $\alpha' \in (0,\alpha]$ are constants depending only on $\delta_0, \eta_0, \alpha, \|f\|_{C^\alpha},
\|g\|_{C^\alpha}.$
Moreover, there exists $\rho_2>0$ small such that $u\in C^{2,\alpha}_{\rm loc}(\mathcal{C}_1\cap B_{\rho_2})\cap C^{2,\alpha'}(\overline{\mathcal{C}_1\cap B_{\rho_2}})$.
\end{theorem}

As a corollary of the two theorems above, we can easily recover (and improve)
the results from \cite{CGN}.

\begin{corollary}\label{c1}
Suppose $X$ and $Y$ are two $C^2$ uniformly convex bounded domains in $\mathbb{R}^n$.  Suppose $\lambda_1$
and $\lambda_2$ are two  positive constants such that $\int_{X}\lambda_1=\int_{Y}\lambda_2.$
 Assume
$f$ and $g$ are two nonnegative densities satisfying
\begin{equation}\label{031}
\int_{X} f(x)dx=\int_{Y}g(y)dy, \quad \|f-\lambda_1\|_{C^{0}(X)}+\|g-\lambda_2\|_{C^{0}(Y)}\leq \delta
\end{equation}
for some $\delta>0$.
Let $u$ be the $c$-convex function associated to the optimal transport problem between $f$ and $g$ with cost 
$c(x,y)$ (see \eqref{eq:Tu}), where $c$ satisfies $(\textbf{C0})$-$(\textbf{C3})$ and 
\begin{equation}\label{902}
\|c+x\cdot y\|_{C^2(X\times Y)} \leq \delta.
\end{equation}
Then, for every 
$\beta\in(0,1)$ there exists $\bar \delta >0$,
depending only on $\beta$, $n$, $\lambda_1$, and $\lambda_2$, such that $u\in C^{1,\beta}(\overline{X})$ provided $\delta \leq \bar \delta$.

In particular, if $c(x,y)=\frac{1}{p}|x-y|^p$ with $p>1$:

$\bullet$  There exists $\bar\delta>0$, depending only on $\beta$, $n$, $p$, ${\rm diam}(X)$, ${\rm diam}(Y)$,  $\lambda_1$, and $\lambda_2$, such that if ${\rm dist}(X,Y)\geq 1/\bar\delta$
then $u\in C^{1,\beta}(\overline{X})$.

$\bullet$ Let $R>0$.  There exists $\bar\delta>0$  depending only on $\beta$, $n$, $p$, ${\rm diam}(X)$, ${\rm diam}(Y)$, $R$, $\lambda_1$, and $\lambda_2$, such that if $2-\bar{\delta}\leq p\leq 2+\bar{\delta}$ and  ${\rm dist}(X,Y)\geq R$, then $u\in C^{1,\beta}(\overline{X}).$
\end{corollary}

\begin{corollary}\label{c2}
In Corollary \ref{c1}, assume that the condition \eqref{031} is replaced by
\begin{equation} \label{032}
\int_{X} f(x)dx=\int_{Y}g(y)dy,\qquad  0<f\in C^{\alpha}(\overline X),\quad 0<g\in C^{\alpha}(\overline Y),
\end{equation}
and that $X,Y$ are of class $C^{2,\alpha}$.
Then there exists $\bar \delta >0$,
depending only on $\beta$, $n$, $\inf f$, $\inf g$,  $\|f\|_{C^\alpha}$, and
$\|g\|_{C^\alpha}$, such that $u\in C^{2,\alpha}_{\rm loc}(X)\cap C^{2,\alpha'}(\overline{X})$ for some $\alpha'\in (0,\alpha)$, provided $\|c(x,y)+x\cdot y\|_{C^2(X\times Y)} \leq \bar \delta$.  

In particular, if $c(x,y)=\frac{1}{p}|x-y|^p$ with $p>1$:

$\bullet$  There exists $\bar\delta>0$, depending only on $\beta$, $n$, $p$, ${\rm diam}(X)$, ${\rm diam}(Y)$, $\inf f$, $\inf g$,  $\|f\|_{C^\alpha}$, and
$\|g\|_{C^\alpha}$,
such that if $ {\rm dist}(X,Y)\geq 1/\bar{\delta}$
then $u\in C^{2,\alpha}_{\rm loc}(X)\cap C^{2,\alpha'}(\overline{X})$ for some $\alpha'\in (0,\alpha)$. 

$\bullet$ Let $R>0$.  There exists $\bar\delta>0$  depending only on $\beta$, $n$, $p$, ${\rm diam}(X)$, ${\rm diam}(Y)$, $R$,  $\inf f$, $\inf g$,  $\|f\|_{C^\alpha},$ and
$\|g\|_{C^\alpha}$,
 such that if $|p-2|\leq \bar{\delta}$ and 
 ${\rm dist}(X,Y)\geq R$ then $u\in C^{2,\alpha}_{\rm loc}(X)\cap C^{2,\alpha'}(\overline{X})$ for some $\alpha'\in (0,\alpha)$.

\end{corollary}
\begin{remark}
In Corollary \ref{c2}, if in addition $X,Y$ are of class $C^\infty$, $f \in C^\infty(\overline X)$, and $g \in C^\infty(\overline Y)$, then $u\in C^\infty(\overline{X}).$
This follows from  the standard regularity theory for linear uniformly elliptic equation with oblique boundary condition, for instance see 
\cite[Theorem 6.31]{GT}. The second part of the corollary follows as in the proof of Corollary \ref{c1}.

\end{remark}

\section{Proof of the corollaries}

\noindent\emph{Proof of Corollary \ref{c1}.}  
Without loss of generality we assume $\lambda_1=\lambda_2=1$ (the general case being completely analogous). Let $v$ be the $c$-convex function associated to the optimal transport problem between 
$\textbf{1}_X$ and $\textbf{1}_Y$ with cost $-x\cdot y$.
Recall that $v$ is of class $C^{2,\alpha}$ up to the boundary (see Theorem \ref{t0}).

Given a point $x_0\in \partial{X}$,  
let $y_0=\nabla v(x_0)$. 
By \cite[Proposition 2.1]{CGN}, after an affine transformation and a translation of coordinates we can assume that $x_0=y_0=\0$, $X\subset \{x^n\geq 0\}$, $Y\subset \{y^n\geq 0\}$, $D^2v(\0)=\text{Id},$ 
and (up to subtracting a constant) $v(\0)=u(\0)=0$.

Now, by \cite[Proposition 4.1]{CGN} (see also \cite[Lemma 4.1]{DF}) we have that 
\begin{equation}\label{901}
\|u-v\|_{L^\infty(X)}\leq \omega(\delta)\rightarrow 0\quad \text{as}\  \delta\rightarrow 0.
\end{equation}
Since $D^2v(\0)=\text{Id},$ 
 and $v$ is of class $C^{2,\alpha}$ up to the boundary, for $h>0$ small the sub-level sets 
of $v$ (recall the notation \eqref{eq:Su}) satisfy
$$ X\cap B_{\frac{2\sqrt{h}}{3}}\subset S_v(h)\subset X\cap B_{\frac{3\sqrt{h}}{2}},$$
and $$ Y\cap B_{\frac{2\sqrt{h}}{3}}\subset \partial^-v(S_v(h))\subset Y\cap B_{\frac{3\sqrt{h}}{2}}.$$
By \eqref{902} and \eqref{901} it is easy to check that, for $h>0$  fixed, provided $\delta$ is sufficiently small
$u$ also satisfies similar properties as follows:
\begin{equation}\label{903}
 X\cap B_{\sqrt{h}/2}\subset S_u(h)\subset X\cap B_{2\sqrt{h}},
 \end{equation}
and 
\begin{equation}\label{904}
Y\cap B_{\sqrt{h}/2}\subset \partial_cu(S_u(h))\subset Y\cap B_{2\sqrt{h}}.
\end{equation}

 Then, we perform
the change of variables
\[
\left\{
\begin{array}
    {l@{\quad}}
   \tilde{x}:=\frac{x}{\sqrt{h}} \\
   \tilde{y}:=\frac{y}{\sqrt{h}}
\end{array}
\right.
\]
and we set
$$\tilde{c}(\tilde{x},\tilde{y}):=\frac1h c(\sqrt{h}\tilde{x}, \sqrt{h}\tilde{y}),\qquad
\tilde{u}(\tilde x):=\frac1h u(\sqrt{h}\tilde{x}).$$
Note that, after this change of variables, $X$ (resp. Y) becomes $\frac{1}{\sqrt{h}} X$ (resp. $\frac{1}{\sqrt{h}}Y)$. Hence the
boundary part $\partial(\frac{1}{\sqrt{h}} X)\cap B_2$ (resp. $\partial(\frac{1}{\sqrt{h}} Y)\cap B_2$) becomes flatter and flatter as $h \to 0$, and in particular it will satisfy \eqref{401} provided $h$ is small enough.  Combining this with
\eqref{902}, \eqref{901}, \eqref{903}, and \eqref{904},  we see that $\tilde{u}$ satisfies all the conditions in Theorem \ref{t1}, hence, $\tilde{u}$
is $C^{1,\beta}$ in a neighborhood of $\0$.  

When initially $x_0$ is in the interior, the argument is similar, the only difference is that instead of using our Theorem \ref{t1} we use its interior 
version by De Philippis and Figalli (see \cite[Theorem 4.3]{DF}). 
Then the proof of the first statement is completed by a standard covering argument. \\

In the case when $c(x,y)=\frac1p|x-y|^p$ with $p>1$ is suffices to observe that,
in both cases, after subtracting $\frac{1}{2}|x|^2+\frac{1}{2}|y|^2$ to $c$ (that does not change the optimal transport problem, see \cite[Section 3.1]{DFreview} or \cite{DF} for more comments on this point) one has
$$
\|c+x\cdot y\|_{C^2(X\times Y)} \to 0 \qquad \text{as $\bar\delta \to 0$}
$$
(see \cite{CGN} for more details). Hence, since $c$ is smooth when $x \neq y$,  the result follows from the first part of the corollary.\qed
\vspace{1em}

\noindent\emph{Proof of Corollary \ref{c2}.} 
We only need to slightly modify the proof of Corollary \ref{c1}.
Let $v$ be the potential function to the optimal transport problem from
$f \textbf{1}_X$ to $g \textbf{1}_Y$   with cost $-x\cdot y$.
Since $f$ and $g$ are of class $C^\alpha$, Caffarelli's boundary $C^{2,\alpha}$ estimate still applies. Using the same argument as 
in the proof of Corollary \ref{c1}, we still have \eqref{902} (when the cost is $c(x,y)=\frac1p|x-y|^p$ with $p>1$), \eqref{901}, \eqref{903}, \eqref{904}, and flatness of the boundary. Hence
all the conditions in Theorem \ref{t2} are satisfied. Therefore $u$ is $C^{2,\alpha}_{\rm loc}\cap C^{2,\alpha'}$ with $\alpha' \in (0,\alpha)$ 
in a small neighborhood of $x_0$, for any $x_0$ on the boundary of $X$. Combining this with the interior $C^{2,\alpha}$ result 
of \cite[Theorem 5.3]{DF} we conclude that $u$ is $C^{2,\alpha}$ in the interior of $X$ and  $C^{2,\alpha'}$ up to the boundary, provided $\delta_0$ is sufficiently small. 
\qed

\section{Proof of the main theorems}

\vspace{1em}

\noindent\emph{Proof of Theorem \ref{t1}.} We divide the proof into several steps.
\vspace{1em}

\noindent$\bullet$ \emph{Step 1: A first change of variables.}
For $x_0\in \mathcal{C}_1\cap B_{\rho_0}$ with $\rho_0\ll1$ to be chosen, take $y_0\in\partial_cu(x_0)$.
Then we perform a change of variables $\bar{x}:=x-x_0$, $\bar{y}:=y-y_0$,
and we define
\begin{equation}\label{e01}
\bar{c}(\bar{x},\bar{y}):=c(x,y)-c(x,y_0)-c(x_0,y)+c(x_0,y_0),
\end{equation}
\begin{equation}\label{e02}
\bar{u}(\bar{x}):=u(x)+c(x,y_0)-c(x_0,y_0)-u(x_0),
\end{equation}
$$\bar{f}(\x):=f(\x+x_0), \qquad \bar{g}:=g(\y+y_0).$$

First we show that, in the new coordinates,
\begin{equation}\label{001}
\|\bar{c}(\bar{x},\bar{y})+\bar{x}\cdot \bar{y}\|_{C^2(\mathcal{C}_1\times\mathcal{C}_2)}\leq 4\delta_0=:\tilde{\delta}_0,
\end{equation}
\begin{equation}\label{002}
\|\bar{u}(\bar{x})-1/2|\bar{x}|^2\|_{L^\infty(\mathcal{C}_1)}\leq K(\sqrt{\eta_0}+\delta_0)=:\tilde{\eta}_0.
\end{equation}

For this, notice that \eqref{001} follows from \eqref{403} easily, so we only need to verify \eqref{002}.
To this aim, we define
 $$
 p_{x_0}:=-D_xc(x_0,y_0)\in \partial^-u(x_0).
 $$
We claim that, for any direction $\textbf{e}$, if $x_0+t\textbf{e}\in \mathcal{C}_1\cap B_{1/2}$ for $0\leq t \leq \sqrt{\eta_0}$ then $(p_{x_0}-x_0)\cdot \textbf{e}\leq K \sqrt{\eta_0}$ for some universal constant $K$. Notice that $u$ is semiconvex, namely, there exists a constant $C$ (depending only on
$\|c\|_{C^2}$) such that $w(x):=u(x)-\frac{1}{2}|x|^2+C|x-x_0|^2$ is convex.
 Since $p_{x_0}-x_0\in \partial^-w(x_0)$,
by convexity and \eqref{eq:uP} we have
\begin{eqnarray*}
(p_{x_0}-x_0)\cdot \textbf{e}&\leq&\frac{w(x_0+\sqrt{\eta_0}\textbf{e})-w(x_0)}{\sqrt{\eta_0}}\\
&\leq& \frac{2\eta_0+C\eta_0}{\sqrt{\eta_0}}=(2+C)\sqrt{\eta_0}.
\end{eqnarray*}
Hence the claim follows with $K:=C+2$.

We now notice that, by \eqref{403},
$$|p_{x_0}-y_0|\leq \|D_xc+\text{id}\|_{L^\infty(\mathcal{C}_1\times\mathcal{C}_2)}\leq \delta_0,$$ therefore,
\begin{equation}\label{501}
(y_0-x_0)\cdot \textbf{e}=(y_0-p_{x_0})\cdot \textbf{e}+ (p_{x_0}-x_0)\cdot \textbf{e}\leq K\sqrt{\eta_0}+\delta_0.
\end{equation}
Now we consider two cases.

\textit{- Case 1: $d(x_0, \{x^n=P(x')\})\geq 2\sqrt{\eta_0}+2\delta_0$.} In this case, we can use any $\textbf{e}\in\mathbb{S}^n$ in \eqref{501},
hence $|y_0-x_0|\leq K\sqrt{\eta_0}+\delta_0$.

\textit{- Case 2:  $d(x_0, \{x^n=P(x')\})\leq 2\sqrt{\eta_0}+2\delta_0$.} For this case, we can still apply \eqref{501} with
$\textbf{e}$ satisfying $\textbf{e}\cdot (0,\cdots,0,1)\geq 1/2$, provided $\delta_0$ is small. 
Combining this with the fact that $y^n\geq -2\delta_0$ (see \eqref{inclusion} and \eqref{401}), we also have
$|y_0-x_0|\leq K(\sqrt{\eta_0}+\delta_0),$ where $K$ needs to be enlarged by a universal constant.

Hence, in both cases 
\begin{equation}
\label{eq:y0}
|y_0-x_0|\leq K\bigl(\sqrt{\eta_0}+\delta_0\bigr),
\end{equation}
and using \eqref{403} and \eqref{eq:uP} we get
\begin{eqnarray*}
\biggl|\bar{u}(\x)-\frac{1}{2}|\x|^2\biggr|&=&\biggl|u(x)-c(x,y_0)+c(x_0,y_0)-u(x_0)-\frac{1}{2}|x-x_0|^2\biggr|\\
&\leq& \biggl|u(x)-\frac{|x|^2}{2}\biggr|+\biggl|u(x_0)-\frac{|x_0|^2}{2}\biggr|\\
  &+&|c(x,y_0)+x\cdot x_0|+|c(x_0,y_0)+x_0\cdot x_0|\\
&\leq& 2\eta+(|x|+|x_0|)|y_0-x_0|+2\delta_0
\leq K(\sqrt{\eta_0}+\delta_0)
\end{eqnarray*}
for some universal constant $K$, as desired.
This concludes the proof of \eqref{002}.

Recall that by assumption $|x_0| \leq \rho_0$ and \eqref{eq:y0} holds, hence (provided $\rho_0$, $\eta_0$, and $\delta_0$ are sufficiently small) we have
$B_{3}(x_0) \supset B_2$, $B_{3}(y_0) \supset B_2$, $B_{1/3}(x_0)\subset B_{1/2}$, 
$B_{1/3}(y_0)\subset B_{1/2}$.
Also, in these new coordinates, the lower part of boundary of $\mathcal{C}_1\cap B_{1/3}$
 (resp. $\mathcal{C}_2\cap B_{1/3}$) is defined by $\bar{P}$ (resp. $\bar{Q}$), and the graph of  $\bar{P}$ (resp. $\bar{Q}$) is only a translation of the graph of
 $P$ (resp. $Q$) by $x_0$ (resp. $y_0$) Notice that, since $x_0$ and $y_0$
 are not necessarily boundary points, it is not true in general that $ \bar P(\textbf{0})=\bar Q(\textbf{0})=0$ nor that $\nabla \bar P(\textbf{0})=\nabla \bar Q(\textbf{0})=\textbf{0}$.
 
%
  \vspace{1em}

\noindent$\bullet$ \emph{Step 2: $\bar{u}$ is close to a strictly convex solution of the Monge-Amp\`ere equation.}
In this step, we approximate $\bar{u}$ by the solution of an optimal transport problem with quadratic cost. This step consists of two lemmas: Lemma
\ref{l0} is about the construction of the approximating solution, and Lemma \ref{l1} is devoted to the smoothness of the approximating solution.
.

\begin{lemma}\label{l0}
Let $\delta>0$, and let $\mathcal{C}_1$ and $\mathcal{C}_2$  be two closed sets such that
\begin{equation}\label{011}
B_{1/K}\cap \{x^n\geq -d_i+\delta\}\subset \mathcal{C}_i\subset B_{K}\cap \{x^n\geq -d_i\},\ \text{for}\ i=1, 2,
\end{equation}
where $0\leq d_i<\frac{1}{10}.$ Let $\tilde{\mathcal{C}}_i:=\mathcal{C}_i\cup (B_{1/K}\cap \{x^n\geq -d_i\}).$
Suppose $f$ and $g$ are two densities supported respectively in $\mathcal{C}_1$ and $\mathcal{C}_2$, and $u:\mathcal{C}_1\rightarrow \mathbb{R}$ is a
$c$-convex functions such that $\partial_cu(\mathcal{C}_1)\subset  B_{K}\cap \{x^n\geq -d_2\}$ and $(T_u)_\sharp f=g.$ Let $\varrho>0$ be such that
$|\tilde{\mathcal{C}}_1|=|\varrho \tilde{\mathcal{C}}_2|$ (where $\varrho \tilde{\mathcal{C}}_2$ denotes the dilation of $\tilde{\mathcal{C}}_2$ with respect
to the origin), and let $v$ be a convex function such that $(\nabla v)_\sharp \textbf{1}_{\tilde{\mathcal{C}}_1} =\textbf{1}_{\varrho \tilde{\mathcal{C}}_2}$
and $v(\0)=u(\0)$. Then there exists an increasing function $\omega: \mathbb{R}^+\rightarrow \mathbb{R}^+,$ depending only on $K$ and satisfying
$\omega(\delta)\geq \delta$ and $\omega(0^+)=0,$ such that, if
\begin{equation}\label{012}
\|f-\textbf{1}\|_{L^\infty({\mathcal{C}_1})}+\|g-\textbf{1}\|_{L^\infty({\mathcal{C}_2})}\leq \delta
\end{equation}
and
\begin{equation}\label{013}
\|c(x,y)+x\cdot y\|_{C^2(\mathcal{C}_1\times\mathcal{C}_2)}\leq \delta,
\end{equation}
then
$$\|u-v\|_{L^\infty(\mathcal{C}_1\cap B_{1/K})}\leq \omega(\delta).$$
\end{lemma}

\begin{proof} The proof of this lemma is similar to the proof of \cite[Lemma 4.1]{DF}. For reader's convenience, we include the details here.
We prove the lemma by contradiction. Assume the lemma is false. Then there exists $\epsilon_0>0$, a sequence 
of closed sets $\mathcal{C}_1^m$, $\mathcal{C}_2^m$ satisfying \eqref{011}, $0\leq d_i^m\leq 1/10$ for $i=1,2$, 
functions $f_m$, $g_m$ satisfying \eqref{012} with
$\delta=1/m$, and costs $c_m$ converging in $C^2$ to $-x\cdot y$, such that
$$u_m(\0)=v_m(\0)=0,\ \text{and}\ \|u_m-v_m\|_{L^\infty(\mathcal{C}^m_1\cap B_{1/K})}\geq \epsilon_0,$$ where 
$u_m$ and $v_m$ are as in the statement.  Note that after passing to a subsequence we can assume
$d_i^m\rightarrow d_i^\infty$ as $m\rightarrow \infty,$ for $i=1,2$. Now we extend $u_m$ and $v_m$ to $B_K$ as 
$$u_m(x):=\underset{z\in \mathcal{C}_1^m, y\in\partial_{c_m}u_m(z)}{\sup}\{u_m(z)-c_m(x,y)+c_m(z,y)\},$$
$$v_m(x):= \underset{z\in \mathcal{C}_1^m, p\in\partial^-v_m(z)}{\sup}\{v_m(z)+p\cdot (x-z)\}.$$
Note that $(T_{u_m})_\sharp f_m=g_m$ gives that $\int f_m=\int g_m,$ so it follows from \eqref{011} and \eqref{012} that
$$\varrho_m=\left(\frac{|\tilde{\mathcal{C}}^m_1|}{|\tilde{\mathcal{C}}^m_2|}\right)^{1/n}\rightarrow 1\quad \text{as}\ m\rightarrow \infty,$$ which implies
that $\partial^-v_m(B_K)\subset B_{\varrho_mK}\subset B_{2K}.$ Thus, since the $C^1$ norm of $c_m$ is uniformly bounded, we deduce that
both $u_m$ and $v_m$ are uniformly Lipschitz. By the assumption that $u_m(\0)=v_m(\0)=0,$ passing to a subsequence, we have that $u_m$ and $v_m$ 
uniformly converge inside $B_K$ to $u_\infty$ and $v_\infty$ respectively, where

\begin{equation}\label{021}
u_\infty(\0)=v_\infty(\0)=0\ \ \ \text{and}\ \ \ \|u_\infty-v_\infty\|_{L^\infty\left(B_{1/K}\cap \{x^n\geq -d^\infty_1\}\right)}\geq \epsilon_0.
\end{equation}

Moreover we have that $f_m$ (resp. $g_m$) weak-$*$ converges in $L^{\infty}$ to some density $f_\infty$ (resp. $g_\infty$). Also, since $\varrho_m\rightarrow {1},$ by \eqref{012} we have that $\textbf{1}_{\mathcal{C}_1}$ (resp. $\textbf{1}_{\varrho \tilde{\mathcal{C}}_2}$)
weak-$*$ converges in $L^\infty$ to $f_\infty$ (resp. $g_\infty$). Note that by \eqref{012} and the fact that $\mathcal{C}_1^m \supset B_{1/K}\cap \{x^n\geq -d^m_i+\delta\}$, we also have that $f_\infty\geq \textbf{1}_{B_{1/K}\cap\{x^n\geq -d_1^\infty\}}.$ 

Now, we apply \cite[Theorem 5.20]{V1} to deduce that both $\nabla u_\infty$ and $\nabla v_\infty$ are optimal transport maps for the cost $-x\cdot y$ sending $f_\infty$ onto
$g_\infty.$ By uniqueness of optimal map, we have that $\nabla u_\infty=\nabla v_\infty$ almost everywhere inside $B_{1/K}\cap\{x^n\geq -d_1^\infty\}\subset
\text{spt}\ f_\infty,$ hence (since $u_\infty(\0)=v_\infty(\0)=0$) $u_\infty=v_\infty$ in $B_{1/K}\cap\{x^n\geq -d_1^\infty\}$, contradicting to \eqref{021}.
\end{proof}

Denote $\tilde{\mathcal{C}}^+_1:= \mathcal{C}_1\cup (B_{1/3}\cap \{\x^n\geq -x_0^n-2\tilde{\delta}_0\})$ and $\tilde{\mathcal{C}}^+_2:=\mathcal{C}_2 \cup(B_{1/3}\cap \{\y^n\geq -y_0^n-2\tilde{\delta}_0\})$ (notice that by \eqref{inclusion} and \eqref{401}, the inclusions $B_{1/3}\cap \{\x^n\geq -x_0^n+2\tilde{\delta}_0\}\subset \mathcal{C}_1\subset B_{3}\cap \{\x^n\geq -x_0^n-2\tilde{\delta}_0\}$ and $B_{1/3}\cap \{\y^n\geq -x_0^n+2\tilde{\delta}_0\}\subset \mathcal{C}_2\subset B_{3}\cap \{\y^n\geq -x_0^n-2\tilde{\delta}_0\}$ hold).

Now let $\varrho$ be such that $|\tilde{\mathcal{C}}^+_1|= |\varrho\tilde{\mathcal{C}}^+_2 |$ (where $\varrho\tilde{\mathcal{C}}^+_2 $ is the dilation of $\tilde{\mathcal{C}}^+_2$ with respect to the origin), and let $v$ be a convex function such that $\nabla v_\sharp \textbf{1}_{\tilde{\mathcal{C}}^+_1}=\textbf{1}_{\tilde{\mathcal{C}}^+_2}$ and $v(\textbf{0})=\bar u(\textbf{0})$=0.
 By \eqref{001} and Lemma \ref{l0}
 \begin{equation}\label{e1}
 \|\bar u-v\|_{L^\infty(\{\x^n\geq \bar{P}\}\cap B_{1/3})}\leq \omega(\tilde\delta_0),
  \end{equation}
  where $\omega:\mathbb{R}^+\rightarrow \mathbb{R}^+$ satisfies $\omega(\delta)\geq \delta$ and $\omega(0^+)=0$.

Next, we use a symmetrization trick. Let $\tilde{\mathcal{C}}_1^-$ (resp. $(\varrho\tilde{\mathcal{C}}_2)^-$) be the reflection of $\tilde{\mathcal{C}}^+_1$ (resp. $\varrho\tilde{\mathcal{C}}^+_2$)  with respect to the hyperplane $\{\x^n=-x_0^n-2\tilde{\delta}_0\}$ (resp. $\{\y^n=\varrho (-y_0^n-2\tilde{\delta}_0)\})$, and denote $\tilde{\mathcal{C}}_1:=\tilde{\mathcal{C}}^+_1\cup \tilde{\mathcal{C}}_1^-$ and
$\tilde{\mathcal{C}}_2:=\varrho\tilde{\mathcal{C}}^+_2\cup (\varrho \tilde{\mathcal{C}}_2)^-$. Let $\tilde{v}$ be the convex potential of the optimal transportation from
$\textbf{1}_{\widetilde{\mathcal{C}}_1}$ to $\textbf{1}_{\widetilde{\mathcal{C}}_2}$, with $\s(\textbf{0})=0$. Then $\tilde{v}$ enjoys the following nice properties.

\begin{lemma}\label{l1}
 $\tilde{v}|_{\tilde{\mathcal{C}}^+_1}=v$, and $\tilde{v}$ is smooth inside $B_{1/10}$ with $\|\tilde{v}\|_ {C^3(B_{1/10})}\leq K$.
\end{lemma}
\begin{proof}
 To prove the claim, it is more convenient to translate both coordinates $\x$ and $\y$ so that both the center of $B_{1/3}\cap \{\x^n= -x_0^n-2\tilde{\delta}_0\}$ and the center of $B_{1/3}\cap\{\y^n=\varrho (-y_0^n-2\tilde{\delta}_0)\})$ are located at the origin. For simplicity, we still use $\x$ and $\y$ to denote the new
variables.

 Then, the first part of the claim follows because the uniqueness of optimal transport map
 implies that it must be symmetric. Indeed, being the densities symmetric  with respect to the hyperplanes $\{\bar x^n=0\}$
 and $\{\bar y^n=0\}$, if $T=(T',T^n):\R^n\to\R^n$ is optimal for the cost $-x\cdot y$
 then the map $\hat T(\bar x',\bar x^n):=\bigl(T'(\bar x',-\bar x^n),-T^n(\bar x',-\bar x^n)\bigr)$
 is still optimal (as it has the same transportation cost of $T$), so by uniqueness $\hat T=T$.
 This proves that $\nabla \tilde v
 |_{\tilde{\mathcal{C}}^+_1}=\nabla v$, and because $\tilde v(\textbf{0})=v(\textbf{0})$ the result follows.
 
 For the second part of the claim notice that, by \eqref{002}
 and 
\eqref{e1}, 
 $$\|\tilde{v}-1/2|\x|^2\|_{L^\infty(B_{1/3}\cap \{ |\x^n|\geq \delta_0\})}\leq \omega(\tilde\delta_0)+\tilde\eta_0.$$  
Now by convexity  of $v$ one can easily show that $\partial^{-} \tilde{v}(B_{1/10})\subset B_{1/6}$,
so arguing as in \cite{FK} we have that $\tilde{v}$ is smooth inside $B_{1/8}$, with $\|\tilde{v}\|_ {C^3(B_{1/8})}\leq K$.
Recalling that $\bar x=x-x_0$ and $\bar y=y-y_0$ with $|x_0| \leq \rho_0$ and $|y_0|\leq \rho_0+K\bigl(\sqrt\eta_0 +\delta_0\bigr)$ (see \eqref{eq:y0}), provided $\rho_0,\delta_0,\eta_0$ are sufficiently small we see that, in the original $(\bar x,\bar y)$ coordinates, $\tilde{v}$ is smooth inside $B_{1/10}$ with $\|\tilde{v}\|_ {C^3(B_{1/10})}\leq K$.
\end{proof}

Next, we compute the hessian of $\tilde{v}$ at $\hat x_1=(0,\cdots, 0, -x^n_0-2\tilde{\delta}_0)$ and the origin.
First note that the $C^3$ bound of $\tilde{v}$ implies that $ \text{Id}/K \leq D^2\tilde{v}\leq K \text{Id}.$
By symmetry we have that
$\nabla_n\tilde{v}$ is constant on $\{\x^n=-x^n_0-2\tilde{\delta}_0\}$, which implies
\begin{equation}\label{m1}
\tilde{v}_{ni}(\hat x_1)=0, \ \text{for}\ 1\leq i\leq n-1.
\end{equation}
 Since $|x_0|\leq \rho_0,$
 by the $C^3$ bound of $\tilde{v}$ we have that
 \begin{equation}\label{m2}
 \s_{ij}(\textbf{0})= \tilde{v}_{ij}(\hat x_1)+O(\rho_0)+O(\tilde{\delta}_0), 1\leq i,j\leq n,
 \end{equation}
so, in particular, $\s_{ni}(\textbf{0})=O(\rho_0)+O(\tilde{\delta}_0)$ for all $i=1,\ldots,n-1$.
\vspace{1em}

\noindent$\bullet$ \emph{Step 3: Initial step for the iteration.} 
In the next Lemma we
will show that there exists an affine transformation such that $\bar u\circ A$
satisfies all the properties in the list:\\1. both the sub-level set $\{\bar{u}\leq h_0\}$ and its image are comparable to $B_{\sqrt{h_0}}$;\\ 2. $\bar{u}$ will be very close to $\frac{|\x|^2}{2}$;\\
 3. both $\|A\|$ and $\|A^{-1}\|$ are bounded by some universal constant.

\begin{lemma}\label{l2}
For every $\tilde{\eta}_0$ small, there exist small positive constants $h_0$, $\tilde{\delta}_0$ for which the following holds:
there exists a symmetric matrix $A$ satisfying $\|A\|$, $\|A^{-1}\|\leq K_1$ and $\det A=1$ such that
\begin{eqnarray*}
A\left(B_{\sqrt{\frac{h_0}{3}}}(\textbf{0}) \right)\cap \{\x^n\geq \bar{P}\} \subset S_{\bar{u}}(h_0):=\{\bar{u}\leq h_0\}\subset
A\left(B_{\sqrt{3h_0}}(\textbf{0})\right)\cap \{\x^n\geq \bar{P}\},\\
A^{-1}\left(B_{\sqrt{\frac{h_0}{3}}}(\textbf{0}) \right)\cap \{\y^n\geq \bar{Q}\} \subset  \partial_{\bar{c}}\bar{u}(S_{\bar{u}}(h_0))\subset
A^{-1}\left(B_{\sqrt{3h_0}}(\textbf{0}) \right)\cap \{\y^n\geq \bar{Q}\}.
\end{eqnarray*}
Moreover $$\biggl\|\bar{u}-\frac{1}{2}|A^{-1}\x|^2\biggr\|_{L^\infty\left(A(B_{\sqrt{3h_0}}(\textbf{0}))\cap\{\x^n\geq \bar{P}\}\right)}\leq \tilde{\eta}_0 h_0,$$
and $A^{-1}(\textbf{e}_n)$ is parallel to $A(\textbf{e}_n)$.
\end{lemma}

\begin{proof}
 First we estimate the norm of $\nabla\tilde{v}(\textbf 0)$. We claim that
\begin{equation}\label{i1}
|\nabla\tilde{v}(\textbf{0})|\leq K_2 \sqrt{\omega(\tilde{\delta}_0)},
\end{equation}
where $K_2$ is a universal constant.
To prove \eqref{i1}, it is enough to show that $$-K_2\sqrt{\omega(\tilde{\delta}_0)}\leq \s_n(\textbf{0})\leq K_2\sqrt{\omega(\tilde{\delta}_0)}\quad
\text{and}\quad -\nabla \s(\0)\cdot \textbf{e}\leq K_2 \sqrt{\omega(\tilde{\delta}_0)},$$ for any unit vector $\textbf{e}$ satisfying $\textbf{e}\cdot \textbf{e}_n\geq 1/2$ .

 Since $\bar{u}$ is semiconvex and $\s$ is smooth inside $B_{1/10}$ with $\|\tilde{v}\|_ {C^3(B_{1/10})}\leq K$, there exists a universal constant
$K_2$ such that $\bar{u}-\s+K_2|\x|^2$ is a convex function inside $B_{1/10}\cap \{\x^n\geq P(\x')\}$. Then by convexity, the fact that $\textbf 0$ is a minimum point for $\bar u$, and \eqref{e1}, we get
\begin{eqnarray*}
-\s_n(\textbf{0})&\leq& \frac{(\bar{u}-\s+K_2|\x|^2)|_{(\textbf{0}+\sqrt{\omega(\tilde{\delta}_0)}\textbf{e}_n)}-(\bar{u}-\s+K_2|\x|^2)|_{\textbf{0}}}
{\sqrt{\omega(\tilde{\delta}_0)}}\\
&\leq& \frac{(K_2+2)\omega(\tilde{\delta}_0)}{\sqrt{\omega(\tilde{\delta}_0)}}=(K_2+2)\sqrt{\omega(\tilde{\delta}_0)},
\end{eqnarray*}
which implies $(-K_2-2)\sqrt{\omega(\tilde{\delta}_0)}\leq\s_n(\textbf{0}).$  By replacing $\textbf{e}_n$ with unit vector $\textbf{e}$ satisfying $\textbf{e}\cdot \textbf{e}_n\geq 1/2$ in the above computation, we also have 
\begin{equation} \label{301}
-\nabla \s(\0)\cdot \textbf{e}\leq K_2 \sqrt{\omega(\tilde{\delta}_0)}.
\end{equation}
Finally, we prove the upper bound on $\s_n(\0)$. Denote by $d_1$ the distance between $\textbf 0$ and $\{\x^n=-x_0^n-2\tilde{\delta}_0\}$.

When $d_1\geq \sqrt{\omega(\tilde{\delta}_0)}$, the proof is the same to the above proof of the lower bound on $\s_n(\0)$ simply replacing $-\textbf{e}_n$ with $\textbf{e}_n$.

When $d_1\leq\sqrt{\omega(\tilde{\delta}_0)}$ we use that $\nabla \tilde v$ maps $\{\x^n=-x_0^n-2\tilde{\delta}_0\}$ onto $\{\y^n=\rho(-y_0^n-2\tilde{\delta}_0\}$
(this follows by the symmetric of $\tilde v$, see the proof of Lemma \ref{l1}) to deduce that $\s_n(0,0,\cdots,-d_1) \leq 0$, hence
\begin{eqnarray*}
\s_n(\textbf{0})&=&\s_n(0,0,\cdots,-d_1)+\biggl(\int_0^1\s_{nn}(0,0,\cdots,-td_1)\,dt \biggr)d_1\\
&\leq& K_2d_1\leq K_2\sqrt{\omega(\tilde{\delta}_0)},
\end{eqnarray*}
concluding the proof of \eqref{i1}.

 By Lemma \ref{l1} and Taylor expansion we have
\begin{eqnarray}
\s(\x)=\nabla \s(\textbf{0})\cdot\x+\frac{1}{2}D^2\s(\0)\x\cdot\x+O(|\x|^3),
\end{eqnarray}
where we used $\s(\textbf{0})=\bar{u}(\textbf{0})=0$, hence
by \eqref{e1} and \eqref{i1} we get
$$\biggl\|\bar{u}-\frac{1}{2}D^2\s(\textbf{0})\x\cdot\x\biggr\|_{L^\infty\left(E(4h_0)\cap\{\x^n\geq \bar{P}\})\right)}
\leq \omega(\tilde{\delta}_0)+K_2 \sqrt{\omega(\tilde{\delta}_0)}\sqrt{h_0}+K_3h_0^{\frac{3}{2}},$$
where $K_2$, $K_3$ are universal constants, and $$E(h_0):=\biggl\{\x: \frac{1}{2}D^2\s(\0)\x\cdot \x\leq h_0\biggr\}.$$
 So if initially we take $\tilde{\delta}_0$, $h_0$ small,
  we have
$$\biggl\|\bar{u}-\frac{1}{2}D^2\s(\0)\x\cdot\x\biggr\|_{L^\infty\left(E(4h_0)\cap \{\x^n\geq \bar{P}\})\right)}
\leq\frac{1}{2} \tilde{\eta}_0 h_0.$$
Denote $A_1:=[D^2\s(\0)]^{-1/2}$. By \eqref{m2} and Lemma \ref{l1}  we see that the angle between $A_1^{-1}(\textbf{e}_n)$ and
$A_1(\textbf{e}_n)$ is bounded by $O(\rho_0)$. Then, it is easy to find a symmetric matrix $A$, with
$\|A-A_1\|=O(\rho_0)+O(\tilde{\delta}_0)$, such that
$A^{-1}(\textbf{e}_n)$ is
parallel to $A(\textbf{e}_n).$ In particular, by choosing $\rho_0$ sufficiently small
we obtain $$\biggl\|\bar{u}-\frac{1}{2}|A^{-1}\x|^2\biggr\|_{L^\infty\left(A(B_{\sqrt{4h_0}}(\textbf{0}))\cap \{\x^n\geq \bar{P}\}\right)}\leq \tilde{\eta}_0 h_0.$$

We now perform a normalization using $A$: Let
\[
\left\{
\begin{array}
    {l@{\quad}}
   \tilde{x}:=A^{-1}\x \\
   \tilde{y}:=A\y
\end{array}
\right.
\]
and $$\tilde{c}(\tilde{x},\tilde{y}):=\bar{c}(A\tilde{x}, A^{-1}\tilde{y}),$$
$$\tilde{u}:=\bar{u}(A\tilde{x}).$$ Note that, since $\text{Id}/K\leq A\leq K\text{Id}$, we have 
\begin{equation}\label{ee1}
\|\tilde{c}+\tilde{x}\cdot \tilde{y}\|_{C^2}\leq K\tilde{\delta}_0
\end{equation}
and
\begin{equation}\label{e5}
\biggl\|\tilde{u}-\frac{1}{2}|\tilde{x}|^2\biggr\|_{L^\infty\left(B_{\sqrt{4h_0}}(\textbf{0})\cap  A^{-1}\{\x^n\geq \bar{P}\}\right)}\leq \tilde{\eta}_0 h_0.
\end{equation}
Let us denote by $d_1'$ (resp. $d_2'$) the distance between $\0$ and the hyperplane $A^{-1}(\{\x^n=-x_0^n-2\tilde{\delta}_0\})$ (resp. $A(\{\y^n=\varrho(-y_0^n-2\tilde{\delta}_0)\}$).
Since by construction $AD^2\tilde v(\textbf 0) A =1+O(\rho_0)+O(\tilde{\delta}_0)$, it follows by 
Lemma \ref{l1} that the hessian of the function $\s(Ax)$ is equal to $(1+o(1))\text{Id}$ inside $B_{K\sqrt{h_0}}$ for some fixed large $K$,
where $o(1) \rightarrow 0$ as $\rho_0,h_0, \tilde{\delta}_0\rightarrow 0$.
Since $\nabla (\tilde v\circ A)$ maps $A^{-1}(\{\x^n=-x_0^n-2\tilde{\delta}_0\})$ onto $A(\{\y^n=\varrho(-y_0^n-2\tilde{\delta}_0\})$,
we deduce that
 $$-d_2'=\partial_n [\s(A\x)]|_\0-(1+o(1))d_1',$$  so, using  \eqref{i1} and the fact that $|x_0|\leq \rho_0 \ll \sqrt{h_0}$,
 \begin{equation}\label{e6}
 |d_1'-d_2'|\leq K_2\sqrt{\omega(\tilde{\delta}_0)}+o(1)\sqrt{h_0}.
 \end{equation}
 By \eqref{e5} and an argument similar to the proof of \eqref{i1},
 one obtains that $\tilde u$ (resp. $\nabla \tilde u$) is close to $\tilde v\circ A$ (resp. $\nabla [\tilde v\circ A])$ and,
 exactly as in the interior case (see \cite[Proof of Theorem 4.3, Step 3]{DF}), we get
 $$B_{\sqrt{\frac{1}{2}h_0}}(\textbf{0})\cap  A^{-1}\{\x^n\geq \bar{P}\}\subset S_{\tilde{u}}(h_0)\subset B_{\sqrt{3h_0}}(\textbf{0})\cap  A^{-1}\{\x^n\geq \bar{P}\},$$ and
$$\partial_{\tilde{c}} \tilde{u}(S_{\tilde{u}}(h_0))\subset B_{\sqrt{3h_0}}(\textbf{0})\cap  A\{\y^n\geq \bar{Q}\}.$$
Now let $\tilde{u}^{\tilde{c}}: B_{\sqrt{4h_0}}(\textbf{0})\cap  A\{\y^n\geq \bar{Q}\}\rightarrow \mathbb{R}$ be a $\tilde{c}^*$-convex function defined by
 $$\tilde{u}^{\tilde{c}}(\tilde{y}):=\displaystyle{\sup_{\tilde{x}\in B_{\sqrt{4h_0}}(\textbf{0})\cap
  A^{-1}\{\x^n\geq \bar{P}\}}}\{-\tilde{c}(\tilde{x},\tilde{y})-\tilde{u}(\tilde{x})\},$$
  where $\tilde c^*(x,y):=\tilde c(y,x)$. Then
by \eqref{ee1}, \eqref{e5}, \eqref{e6}, we have 
\begin{eqnarray*}
\biggl\|\tilde{u}^{\tilde{c}}-\frac{1}{2}|\tilde{y}|^2\biggr\|_{B_{\sqrt{4h_0}}(\textbf{0})\cap  A\{\y^n\geq \bar{Q}\}}&\leq& \tilde{\eta}_0h_0+K\tilde{\delta}_0+\Bigl(K_2\sqrt{\omega(\tilde{\delta}_0)}+o(1)\sqrt{h_0}\Bigr)\sqrt{h_0}\\
&\leq& 2\tilde{\eta}_0h_0,
\end{eqnarray*}
provided $\tilde{\delta}_0$, $\rho_0$, and $\sqrt{h_0}$ are small enough. Also, similarly to above,
$$\partial_{\tilde{c}^*} \tilde{u}^{\tilde{c}}\left({B_{\sqrt{\frac{1}{3}h_0}}(\textbf{0})\cap  A\{\y^n\geq \bar{Q}\}}\right)\subset B_{\sqrt{\frac{1}{2}h_0}}(\textbf{0})\cap  A^{-1}\{\x^n\geq \bar{P}\}\subset S_{\tilde{u}}(h_0).$$
Therefore ${B_{\sqrt{\frac{1}{3}h_0}}(\textbf{0})\cap  A\{\y^n\geq \bar{Q}\}}\subset \partial_{\tilde{c}}\tilde{u}(S_{\tilde{u}}(h_0))$, and translating back to the $\x, \y$ coordinates this completes the proof of Lemma \ref{l2}.
\end{proof}
\vspace{1em}

\noindent$\bullet$ \emph{Step 4: The iteration argument.}
We begin by noticing that by construction $\bar{d}_1:=\text{dist}(\0, B_{1/3}\cap\{\x^n=\bar{P}\})\leq \rho_0$.

Up to an affine transformation we can assume $D_{xy}\bar{c}(\0, \0)=\text{Id}.$ We now perform a change of variable: Let
\[
\left\{
\begin{array}
    {l@{\quad}}
   \tilde{x}:=\frac{1}{\sqrt{h_0}}A^{-1}\x \\
   \tilde{y}:=\frac{1}{\sqrt{h_0}}A\y
\end{array}
\right.
\]
and $$c_1(\tilde{x},\tilde{y}):=\frac{1}{h_0}\bar{c}(\sqrt{h_0}A\tilde{x}, \sqrt{h_0}A^{-1}\tilde{y}),$$
$$u_1:=\frac{1}{h_0}\bar{u}(\sqrt{h_0}A\tilde{x}),$$ where $A$ is from Lemma \ref{l2}.
Note that, since $A^{-1}(\textbf{e}_n)$ is parallel to $A(\textbf{e}_n)$, after the transformation, we have that $\{\x^n=\bar{P}(\x')\}$ (resp. $\{\y^n=\bar{Q}(\y')\}$) becomes $A^{-1}\{\x^n=\bar{P}(\x')\}$ (resp. $A\{\y^n=\bar{Q}(\y')\}$),
and after a rotation of coordinates it can be written as $\{\tilde{x}^n=P_1(\tilde{x}')\}$ (resp. $\{\tilde{y}^n=Q_1(\tilde{y}')\}$). Since
$(\|A\|+\|A^{-1}\|)\sqrt{h_0}\ll 1$, we can ensure that
$\|P_1\|_{C^2}+\|Q_1\|_{C^2}\leq \tilde{\delta}_0.$  


We also define $$f_1(\tilde{x}):=\bar{f}(\sqrt{h_0}A\tilde{x}),\quad g_1(\tilde{y}):=\bar{g}(\sqrt{h_0}A^{-1}\tilde{y}).$$
Since $\det(A)=1$, we have that $(T_{u_1})_\sharp f_1=g_1.$ Moreover, defining
$$\mathcal{C}_1^{(1)}:=S_{u_1}(1),\quad \mathcal{C}_2^{(1)}:=\partial_{c_1}u_1(S_{u_1}(1)),$$
it follows by Lemma \ref{l2} that
\begin{equation}
 B_{1/3}\cap\{\tilde{x}^n\geq P_1(\tilde{x}')\}\subset \mathcal{C}_1^{(1)}\subset B_{3}\cap\{\tilde{x}^n\geq P_1(\tilde{x}')\},
\end{equation}
\begin{equation}
B_{1/3}\cap\{\tilde{y}^n\geq Q_1(\tilde{x}')\}\subset \mathcal{C}_2^{(1)}\subset B_{3}\cap\{\tilde{y}^n\geq Q_1(\tilde{x}')\}.
\end{equation}
Now it is easy to check that $u_1, c_1, f_1, g_1, \mathcal{C}_1^{(1)},  \mathcal{C}_2^{(1)}$ satisfy all the conditions for Lemma \ref{l2}.
Therefore, we can apply Lemma \ref{l2} to $u_1$ and we can find a matrix $A_1$ satisfying
$\|A_1\|, \|A_1^{-1}\|\leq K_1$, $\det(A_1)=1,$
\begin{equation}
A_1\left(B_{\sqrt{\frac{h_0}{3}}}(\textbf{0}) \right)\cap \{\tilde{x}^n\geq P_1\} \subset S_{u_1}(h_0):=\{u_1\leq h_0\}\subset
A_1\left(B_{\sqrt{3h_0}}(\textbf{0})\right)\cap \{\tilde{x}^n\geq P_1\},\label{eq:shape1}
\end{equation}
\begin{equation}
A_1^{-1}\left(B_{\sqrt{\frac{h_0}{3}}}(\textbf{0}) \right)\cap \{\tilde{y}^n\geq Q_1\} \subset  \partial_{c_1}u_1(S_{u_1}(h_0))\subset
A_1^{-1}\left(B_{\sqrt{3h_0}}(\textbf{0}) \right)\cap \{\tilde{y}^n\geq Q_1\},\label{eq:shape2}
\end{equation}
$$\biggl\|u_1-\frac{1}{2}|A_1^{-1}\tilde{x}|^2\biggr\|_{L^\infty\left(A_1(B_{\sqrt{3h_0}}(\textbf{0}))\cap\{\tilde{x}^n\geq P_1\}\right)}\leq \tilde{\eta}_0 h_0,$$
and  $A_1^{-1}(\textbf{e}_n)$ is parallel to $A_1(\textbf{e}_n)$. Note that after a rotation we can just assume $A_1^{-1}(\textbf{e}_n)$ and $A_1(\textbf{e}_n)$ are in $\textbf{e}_n$ direction.

Now, we finally fix $\rho_0: =10K\sqrt{h_0}$ (where $K$ is
constant in Lemma \ref{l1}).
 As long as $\text{dist}(\0, B_{1/3}\cap\{\tilde{x}^n=P_k\})\leq \rho_0,$ , we can continue the iteration, hence we only need to consider two cases.
\smallskip

-\emph{Case 1}: At step $k+1$, $d(\0, B_{1/3}\cap\{\tilde{x}^n=P_{k+1}\})> \rho_0.$ Here we assume $k+1$ is the smallest among all such integers.
\smallskip

-\emph{Case 2}: The iteration goes on forever. Note that case 2 only happens when $\0$ is
on the boundary of $\mathcal{C}_1.$
\smallskip

For Case 1, at step $k+1$, after an affine transformation $M_k:=A_{k-1}\cdots A_1$ we have that
$$u_{k+1}:=\frac{1}{h_0^k}u_1(h_0^{k/2}M_k\tilde{x}),\qquad
c_{k+1}:=\frac{1}{h_0^k}c_1(h_0^{k/2}M_k\tilde{x}, h_0^{k/2}M'^{-1}_k\tilde{y}),
$$
$$f_{k+1}:=f_1(h_0^{k/2}M_k\tilde{x}),\qquad
g_{k+1}:=g_1(h_0^{k/2}M'^{-1}_k\tilde{y})
$$ satisfy the same conditions as $u_1, c_1, f_1, g_1$ with exactly the same constants (here and in the sequel,  $M_k'$ denotes the transpose of $M_k$).
 Since $\text{dist}(\0, B_{1/3}\cap\{\tilde{x}^n=P_{k+1}\})\geq \rho_0=10K\sqrt{h_0},$  
 by doing one more rescaling we obtain that $\text{dist}(\0, B_{1/3}\cap\{\tilde{x}^n=P_{k+2}\}) \geq 1/K'$
 for some $K'>0$ universal,
 so we have reduced ourselves to the interior problem as the one studied in \cite{DF}.
 In particular, by \cite[Theorem 4.3]{DF} we obtain
that $u$ is $C^{1,\beta}$ at $\0$.

For Case 2, with the same notation as in Case 1 we have that, for each $k\geq1$,
\begin{equation}\label{e31}
\text{Id}/K_1^k\leq M_k\leq K_1^k\text{Id},\quad \det{M_k}=1,
\end{equation}
\begin{equation}\label{e32}
M_k\left(B_{\frac{1}{3}h_0^{{k}/{2}}}\right)\cap\{\tilde{x}^n\geq P_1\}\subset S_{u_1}(h_0^k)\subset M_k\left(B_{3h_0^{{k}/{2}}}\right)\cap\{\tilde{x}^n\geq P_1\}.
\end{equation}
\begin{equation}\label{e51}
M'^{-1}_k\left(B_{\frac{1}{3}h_0^{{k}/{2}}}\right)\cap\{\tilde{y}^n\geq Q_1\}\subset \partial_{c_1}u_1\left(S_{u_1}(h_0^k)\right)\subset M'^{-1}_k\left(B_{3h_0^{{k}/{2}}}\right)\cap\{\tilde{y}^n\geq Q_1\}.
\end{equation} 
By \eqref{e31} we have
\begin{equation}
B_{\frac{1}{3}\left(\frac{\sqrt{h_0}}{K_1}\right)^{k}}\cap\{\tilde{x}^n\geq P_1\}\subset S_{u_1}(h_0^k)\subset B_{3(K_1\sqrt{h_0})^{k}}\cap\{\tilde{x}^n\geq P_1\}
\end{equation}
and
\begin{equation}\label{e52}
B_{\frac{1}{3}\left(\frac{\sqrt{h_0}}{K_1}\right)^{k}}\cap\{\tilde{x}^n\geq Q_1\}\subset \partial_{c_1}u_1\left(S_{u_1}(h_0^k)\right)\subset B_{3(K_1\sqrt{h_0})^{k}}\cap\{\tilde{y}^n\geq Q_1\},
\end{equation}
so defining $r_0:=\frac{\sqrt{h_0}}{3K_1}$ we obtain, for $\beta<1$ fixed,
$$\|u_1\|_{L^\infty\left(B_{r_0^k}\cap \{\tilde{x}^n\geq P_1\}\right)}\leq h_0^k=(3K_1r_0)^{2k}\leq r_0^{(1+\beta)k},$$ provided $h_0$ (and so $r_0$)
is sufficiently small. This implies the $C^{1,\beta}$ regularity of $u_1$ at $\0$,
which means that $u$ is $C^{1,\beta}$ at $x_0$. Since $x_0 \in \mathcal C_1\cap B_{\rho_0}$
is arbitrary, this concludes the proof of the theorem.
\qed
\vspace{1em}

\begin{remark}\label{rmk1}
Under the conditions of Theorem \ref{t1}, the following useful property holds: there exists $\rho_1\leq \rho_0$ such that
$T_u(B_{\rho_1}\cap \{x^n=P(x')\})\subset\{y^n=Q(y')\}$.  Indeed, let
$$u^c(y)=\underset{x\in \mathcal{C}_1}{\sup}\{-c(x,y)-u(x)\}.$$
By \eqref{401}, \eqref{403}, and \eqref{eq:uP}
it is easy to check that $\|u^c(y)-\frac{1}{2}|y|^2\|_{C^0(B_{1/2}\cap\{y^n>Q(y')\})}\rightarrow 0$ as $\delta_0\rightarrow 0.$ Hence, when $\delta_0$ is sufficiently small,
by restricting to a smaller domain we can still apply Theorem \ref{t1} to $u^c$ obtaining that $u^c$ is $C^{1,\beta}$ in $B_{\rho_0'}\cap \{y^n\geq Q(y')\}$ for some $\rho_0'>0$. Let $T_{u^c}$ be the optimal transport map from $\mathcal{C}_2$ to $\mathcal{C}_1$.
It is well known that $Du^c(y)=D_yc(T_{u^c}(y),y)$, and $T_{u^c}$ is the inverse
of $T_u$ in an almost everywhere sense. Since $u$ (resp. $u^c$) is $C^{1,\beta}$ in $B_{\rho_0}\cap \{x^n\geq P(x')\})$ (resp. $B_{\rho_0'}\cap \{y^n\geq Q(y')\}$) we deduce that both $T_u$ and $T_{u^c}$
are continuous near $\0$, being one the inverse of the other, $T_u$ is a homeomorphism from $B_{\rho_1}\cap \{x^n\geq P(x')\}$ to
$T_u(B_{\rho_1}\cap \{x^n\geq P(x')\})$ for any $\rho_1$ sufficiently small. From this fact it is easy to conclude that $T_u(B_{\rho_1}\cap \{x^n=P(x')\})\subset\{y^n=Q(y')\}$, as desired.
\end{remark}

\section{$C^{2,\alpha}$ regularity}
Below we still use $P$ and $Q$ to denote two $C^2$ functions defined on $\mathbb{R}^{n-1}$ such that $P(\textbf{0})=Q(\textbf{0})=0$, $\nabla P(\textbf{0})=\nabla Q(\textbf{0})=\textbf{0}$, and
\begin{equation}\label{e11}
\|P\|_{C^2}+ \|Q\|_{C^2}\leq \delta.
\end{equation}
$co[E]$ is used to denote the convex hull of a set $E$, and $\mathcal{N}_r(E)$ is used to denote the $r$-neighborhood of $E$. $S^-$ denotes the reflection of $S$ with respect to the hyperplane $\{x^n=0\},$ and $\tilde{S}:=S\cup S^-$.

\begin{lemma} (Comparison principle) \label{l3} Let $u$ be a $c$-convex function of class $C^1$ inside the set $S:=\{u<1\}$, and assume that $u(\textbf{0})=0$,
\begin{equation}\label{e12}
B_{1/K}\cap \{x^n\geq P(x')\}\subset S\subset B_K\cap\{x^n\geq P(x')\},
\end{equation}
 \begin{equation}\label{e13}
 B_{1/K}\cap \{y^n\geq Q(y')\}\subset \partial_c u(S)\subset B_{K}\cap \{y^n\geq Q(y')\},
\end{equation}
and 
\begin{equation}\label{bb}
\partial_c u\left(\partial S\cap \{x^n=P(x')\}\right)\subset B_{K}\cap \{y^n= Q(y')\}.
\end{equation}

  Let $f, g$ be two densities such that $$\|f/\lambda_1-1\|_{L^\infty(S)}+\|g/\lambda_2-1\|_{L^\infty(S)}\leq \epsilon$$ for some constants
$\lambda_1, \lambda_2\in(1/2,2)$ and $\epsilon\in(0,1/4),$ and assume that $(T_u)_\sharp f=g.$ Furthermore, suppose that
\begin{equation}\label{e15}
\|c+x\cdot y\|_{C^2(S\times\partial_cu(S) )}\leq \delta.
\end{equation}

Then there exist a universal constant $\gamma\in (0,1)$ and $\delta_1=\delta_1(K)>0$ small, such that the following holds: Let $v$ be the solution of
\begin{equation}
\label{eq:MA v}
\begin{cases}
\det(D^2v)=\lambda_1/\lambda_2 &\mbox{in $\mathcal{N}_{\delta^\gamma}(co[\tilde{S}])$},\\
v=1&\mbox{on $\partial\left(\mathcal{N}_{\delta^\gamma}(co[\tilde{S}])\right)$}
 \end{cases}
 \end{equation}
Then
\begin{equation}
\|u-v\|_{L^\infty(S)}\leq C_K(\epsilon+\delta^{\gamma/n})\ \ \text{provided}\ \delta\leq \delta_1,
\end{equation}
where $C_K$ is a constant independent of $\lambda_1, \lambda_2, \epsilon$, and $\delta$.
\end{lemma}

Note that in our application of the comparison principle in the proof of Theorem \ref{t2}, condition \eqref{bb} follows from Remark \ref{rmk1}.
\vspace{1em}

\emph{Proof.} First of all we notice that, for any $\gamma\in(0,1)$, we have that
\begin{equation}\label{e61}
{\rm dist}\left(x, \partial\bigl(\mathcal{N}_{\delta^\gamma}(co[\tilde{S}])\bigr)\right)\leq C_K'\delta^{\gamma},\ \text{for any}\ x\in \partial S\cap \{x^n> P(x')\}.
\end{equation}
Indeed, this follows by the very same argument as the one at the beginning of the proof of
\cite[Proposition 5.2]{DF}, where the same estimate is proved in a similar situation.

Now, by standard interior estimates for solution of the Monge-Amp\`ere equation with constant right hand side, we also have
\begin{equation}\label{e16}
\text{osc}_Sv\leq C_K'',
\end{equation}
\begin{equation}\label{e17}
1-C_K''\delta^{\gamma/n}\leq v\leq 1\quad \text{on}\ \partial S\cap \{x^n>P(x')\},
\end{equation}
\begin{equation}\label{e18}
C_K''\delta^{-\frac{(n-1)\gamma}{\tau}}\geq D^2v\geq \delta^{\gamma/\tau}\text{Id}/C_K''\quad \text{in}\ co[\tilde{S}],
\end{equation}
for some $\tau>0$ universal and some constant $C_K''$ depending only on $K$.
For any point $x$ on $\partial S\cap \{x^n= P(x')\}$, by \eqref{e11},
\eqref{bb}, and \eqref{e15} we have
\begin{equation}\label{e21}
\biggl|\frac{\partial{u}}{\partial{x^n}}\biggr|\leq C_K''\delta,
\end{equation}

Since $\mathcal{N}_{\delta^\gamma}(co[\tilde{S}])$ is symmetric with respect to $\{x^n=0\},$ we see that $v$ satisfies $v(x',x^n)=v(x',-x^n)$, which implies
$\frac{\partial v}{\partial x^n}=0$ on $co[\tilde{S}]\cap\{x^n=0\}$.
Now, for $x=(x', x^n)\in \partial S\cap \{x^n= P(x')\}$, we take the point $z:=(x', 0)\in co[\tilde{S}]\cap\{x^n=0\}$, by \eqref{e18} we have that
\begin{equation}\label{e22}
0\leq \frac{\partial v}{\partial x^n}(x)\leq \frac{\partial v}{\partial x^n}(z)+C_K''\delta^{-\frac{(n-1)\gamma}{\tau}}\delta\leq C_K''\delta^{\gamma/n},
\end{equation}
where we choose $\gamma\leq \frac{\tau}{4}$ small so that $\gamma/n\leq 1-\frac{(n-1)\gamma}{\tau}.$

Let us define
$$v^+:=(1+4\epsilon+2\sqrt{\delta})v-4\epsilon-2\sqrt{\delta}+4C_K''\delta^{\gamma/n}(x^n-K),$$ and
$$v^-:=\Bigl(1-4\epsilon-\frac{\sqrt{\delta}}{2}\Bigr)v+4\epsilon+\frac{\sqrt{\delta}}{2}+4C_K''\delta^{\gamma/n}(-x^n+K+1).$$
First, it is easy to check that $v^-\geq u\geq v^+$ on $\partial S\cap\{x^n>P(x')\}.$
Also, by \eqref{e21} and \eqref{e22} we have
\begin{equation}\label{e23}
\frac{\partial v^+}{\partial x^n}> \frac{\partial u}{\partial x^n}>\frac{\partial v^-}{\partial x^n}
\end{equation}
on $\partial S\cap\{x^n=P(x')\}.$
To prove the lemma, we need only to show that $v^-\geq u\geq v^+$ on $S$. In fact, if $u\geq v^+$ fails, then $\max (v^+-u)>0$ is achieved at some point
$z\in S$. If $z\in \partial S\cap\{x^n=P(x')\}$, then we move the graph of $v^+$ down and then lift it up, it will touch the graph of $u$ at point
$(z,u(z))$ from below, which is contradict to \eqref{e23}. Therefore, $z$ must be an interior point of $S$. Hence, we can find a number $\eta>0$ such that
$\{v^+-\eta-u\geq 0\}\subset\subset S$, and
using a maximum principle argument for the equations \eqref{eq:MA T} and \eqref{eq:MA v}  we can reach a contradiction as same as the proof of \cite[Proposition 5.2]{DF}. The other part is similar.
\qed

\vspace{1em}

\noindent  \emph{Proof of Theorem \ref{t2}.}
Fixed a point $x_0\in B_{\rho_1/2}\cap \{x^n=P(x')\}$, take $y_0:=\cexp_{x_0}(\nabla u(x_0))\in \partial_cu(x_0)$ (notice that $u$ is $C^1$ at the boundary by Theorem \ref{t1}).
Up to a change of variable as in the proof of Theorem \ref{t1}, we can assume that $x_0=y_0=\0$, $u\geq 0$, $u(\0)=0,$ and $D_{xy}c(\0,\0)=\text{Id}.$

We set $S_h:=S_u(h)=\{u\leq h\}$.
\vspace{1em}

\noindent$\bullet$ \emph{Step 1: $C^{1,1}$ estimate at $\0$.}  We show that
\begin{equation}\label{e44}
B_{\sqrt{h}/K}\cap \{x^n\geq P(x')\} \subset S_h \subset B_{K\sqrt{h}}\cap  \{x^n\geq P(x')\}\ \ \ \forall\,h\leq h_1
\end{equation}
for some $h_1$ and $K$ universal.

To prove this fact we begin by observing that, 
by \eqref{e43}, for any $h_1>0$ we can choose $\eta_1=\eta_1(h_1)>0$ small enough such that \eqref{e44} holds for $S_{h_1}$ with $K=2$.
Hence, assuming without loss of generality that $\delta_1\leq 1$, we have that
$$B_{\frac{\sqrt{h_1}}{3}}\subset \mathcal{N}_{\delta_1^\gamma\sqrt{h_1}}(co[\tilde{S}_{h_1}])\subset B_{3\sqrt{h_1}},$$
where
$\gamma$ is the exponent from Lemma \ref{l3}, and $\tilde{S}_{h_1}$ is constructed in the same way as $\tilde{S}$ in Lemma \ref{l3}.


  Let $v_1$ solve the following Monge-Amp\`ere equation

\begin{equation}
\begin{cases}
\det(D^2v)=f(\0)/g(\0) &\mbox{in $\mathcal{N}_{\delta_1^\gamma\sqrt{h_1}}(co[\tilde{S}_{h_1}])$},\\
v=1&\mbox{on $\partial\left(\mathcal{N}_{\delta_1^\gamma\sqrt{h_1}}(co[\tilde{S}_{h_1}])\right)$}.
 \end{cases}
 \end{equation}
Since $B_{1/3}\subset  \mathcal{N}_{\delta_1^\gamma\sqrt{h_1}}(co[\tilde{S}_{h_1}])/\sqrt{h_1}\subset B_3$, by standard Pogorelov estimates we have that
$|D^2v_1(\0)|\leq M$, where $M>0$ is a large universal constant.

Now we recall a useful fact for the standard Monge-Amp\`ere equation. Let $w$ be a solution of
\begin{equation}
\label{eq:MA}
\begin{cases}
\det(D^2w)=f(\0)/g(\0) &\mbox{in }Z,\\
w=1&\mbox{on $\partial Z$},
 \end{cases}
 \end{equation}
where $\0\in Z$ is a convex set, $ -1<\inf w\leq w(\0)<1/2$,
and $|D^2w(\0)|\leq M+1$.
Then there exists a large universal $\bar{K}$ such that
$B_{1/\bar{K}}\subset Z\subset B_{\bar{K}}$. For reader's convenience, we give the proof below.

\begin{proof}
 By John's lemma we can find a matrix $A$ with $A(\0)=\0$ and $\det(A)=1$ such that $B_{1/C_1}(z)\subset A(Z)\subset B_{C_1}(z)$, where $C_1$ is a universal constant.
Then it is easy to check that $\bar{w}(x)=w(A^{-1}x)$ solves
\begin{equation}
\label{eq:MA2}
\begin{cases}
\det(D^2\bar{w})=f(\0)/g(\0) &\mbox{in }A(Z),\\
\bar{w}=1&\mbox{on $\partial (A(Z))$}.
 \end{cases}
 \end{equation}
Since $ -1<\inf \bar{w}\leq \bar{w}(\0)<1/2 $, it follows that ${\rm dist}(\0, \partial A(Z))> c$,  where $c>0$ is a universal constant, hence
$B_c\subset A(Z)\subset B_{C_1}$.
By interior estimates for the standard Monge-Amp\`ere equation, we have that $|D^2\bar{w}(\0)|\leq C_2$ for some universal constant $C_2$.
Since $D^2\bar{w}(\0)=A^{-1}D^2w(\0)A^{-1}$, $|D^2\bar{w}(\0)|\leq C_2$ and $|D^2w(\0)|\leq M+1$, it follows 
from \eqref{eq:MA} and \eqref{eq:MA2}
 that $\|A\|, \|A^{-1}\|\leq C_3$ for some universal
constant $C_3$ (recall that by assumption $f(\0)/g(\0)$ is bounded away from zero and infinity).
Since $A^{-1}(B_c)\subset Z\subset A^{-1}(B_{C_1})$ and $A(\0)=\0$, it follows that
there exists a universal constant $\bar{K}$ such that
$B_{1/\bar{K}}\subset Z\subset B_{\bar{K}}$, as desired.
\end{proof}

Now, we prove by induction that \eqref{e44} holds with $K:=2\bar{K}$. Let $h_k:=h_12^{-k}.$ If $h=h_1$ then we already know that \eqref{e44} holds with $K=2.$
Assume that \eqref{e44} holds with $h=h_k$ and $K=2\bar{K}$, and we will show that it holds with $h=h_{k+1}.$ For this, for any $k\in \mathbb{N}$ let $v_k$
be the solution of

\begin{equation}
\begin{cases}
\det(D^2v_k)=f(\0)/g(\0) &\mbox{in $\mathcal{N}_{\delta^\gamma\sqrt{h_k}}(co[\tilde{S}_{h_k}])$},\\
v_k=1&\mbox{on $\partial\left(\mathcal{N}_{\delta^\gamma\sqrt{h_k}}(co[\tilde{S}_{h_1}])\right)$},
 \end{cases}
 \end{equation}
 and $\delta_k:=\|c(x,y)+x\cdot y\|_{C^2\left(S_{h_k}\times T_u(S_{h_k})\right)}+ \|P_k\|_{C^{2}}+\|Q_k\|_{C^{2}},$ where $P_k=\frac{1}{\sqrt{h_k}}P(\sqrt{h_k}x')$ and $Q_k=\frac{1}{\sqrt{h_k}}Q(\sqrt{h_k}y').$ Then it is easy to see that
 $\|P_k\|_{C^{2}}+\|Q_k\|_{C^{2}}\leq C\sqrt{h_k}.$
By the $C^{1,\beta}$ regularity of $u$ (which implies that $T_u$ is $C^{0,\beta}$) we also have
 $$\|c(x,y)+x\cdot y\|_{C^2\left(S_{h_k}\times T_u(S_{h_k})\right)}\leq C\Bigl(\text{diam}(S_{h_k})+\text{diam}(T_u(S_{h_k}))\Bigr)\leq Ch_k^{\beta/2},$$
 which implies in particular that
\begin{equation}
\label{eq:delta h}
\delta_k\leq Ch_k^{\beta/2}.
\end{equation}
 Let us consider the rescaled functions
 $$\bar{u}_k(x):=u(\sqrt{h_k}x)/h_k,\quad \bar{v}_k(x):=v_k(\sqrt{h_k}x)/h_k$$
 and notice that,
by the inductive hypothesis,
\begin{equation}\label{e71}
B_{\frac{1}{2\bar{K}}}\cap\{x^n\geq P_k(x')\} \subset \bar{S}_k:=\{\bar{u}_k\leq 1\}\subset B_{2\bar{K}}\cap\{x^n\geq P_k(x')\}.
\end{equation}

Note that by \eqref{e32}, \eqref{e51} we have that there exists an affine transformation $L_k$ 
such that both $L_k(\bar{S}_k)$ and  $L_k'^{-1}\left(\partial_cu(\bar{S}_k)\right)$ are universally comparable to half-balls. 
By \eqref{e71}, we see that $\bar{S}_k$ is already comparable to a half-ball, hence $L_k$ satisfies $\|L_k\|, \|L_k^{-1}\|\leq K'$ for
some universal constant $K'$, which implies that $\partial_cu(\bar{S}_k)=L'_k\bigl(L_k'^{-1}\left(\partial_cu(\bar{S}_k)\right) \bigr)$ is also universally comparable to a
half-ball,
that is, there exists a universal constant $\bar K' \geq \bar K$ such that
$$
B_{\frac{1}{2\bar{K}'}}\cap\{y^n\geq Q_k(y')\} \subset \partial_cu(\bar{S}_k)\subset B_{2\bar{K}'}\cap\{y^n\geq Q_k(y')\}.
$$
This estimate and \eqref{e71} allow us to 
apply Lemma \ref{l3} and deduce that
\begin{equation}\label{e73}
\|\bar{u}_k-\bar{v}_k\|_{L^\infty(\bar{S}_k)}\leq C_{\bar{K}'}\biggl(\underset {S_{h_k}}{\text{osc}}f+\underset {T_u(S_{h_k})}{\text{osc}}g+\delta_k^{\gamma/n}\biggr)
\leq C_{\bar{K}'}h_k^{\frac{\alpha\beta\gamma}{2n}}.
\end{equation}
By \eqref{e71} we see that when $h_1$ is small $\{\bar{v}_k\leq 1\}$ has ``good shape'', namely
\begin{equation}
\label{eq:Sv}
B_{\frac{1}{3\bar{K}}\sqrt{h_k}}\subset \mathcal{N}_{\delta_k^\gamma\sqrt{h_k}}(co[\tilde{S}_{h_k}])\subset B_{3\bar{K}\sqrt{h_k}}.
\end{equation}
Now we show that
 \begin{equation}\label{e81}
 2C_{\bar{K}'}h_k^{\frac{\alpha\beta\gamma}{2n}}>\displaystyle{\inf_{\mathcal{N}_{\delta^\gamma\sqrt{h_k}}(co[\tilde{S}_{h_k}])}}\bar{v}_k\geq -2C_{\bar{K}'}h_k^{\frac{\alpha\beta\gamma}{2n}},
 \end{equation}
 provided $h_1$ is small enough.
By convexity and symmetry, we know that $\displaystyle{\inf_{\mathcal{N}_{\delta^\gamma\sqrt{h_k}}(co[\tilde{S}_{h_k}])}}\bar{v}_k$ is achieved at $\{x^n=0\}$. Let us denote
$$
\hat \pi:\mathbb{R}^n\rightarrow \mathbb{R}^{n-1}:=\{x^n=0\}
$$
as the standard projection, and for any
$(z',0)\in \hat \pi\left(\{\bar{u}_k=1\}\cap\{x^n=P_k(x')\}\right)$ we denote the corresponding point on $\{\bar{u}_k=1\}\cap\{x^n=P_k(x')\}$ as $z:=(z', z^n).$
By \eqref{e22}, \eqref{eq:delta h}, and \eqref{e73}, we have that 
$$\bar{v}_k(z',0)\geq \bar{u}_k(z)- C_{\bar{K}'}h_k^{\frac{\alpha\beta\gamma}{2n}}-C''_{\bar{K}}\delta_k^{\gamma/n}\sqrt{h_k}\geq 1/2,$$ provided $h_1$ small. Also, by \eqref{e73}
we also have $\bar{v}_k(\0)\leq  C_{\bar{K}'}h_k^{\frac{\alpha\beta\gamma}{2n}}$ (recall that $u(\0)=0$, so $\displaystyle{\inf_{\mathcal{N}_{\delta^\gamma\sqrt{h_k}}(co[\tilde{S}_{h_k}])}}\bar{v}_k$ is achieved inside  $\hat \pi\left(\{\bar{S}_k\cap\{x^n=P_k(x')\}\right)$. Again by \eqref{e22}, \eqref{eq:delta h}, and \eqref{e73}, we have 
$$\bar{v}_k(z',0)\geq \bar{u}_k(z)-C_{\bar{K}'}h_k^{\frac{\alpha\beta\gamma}{2n}}-C''\delta_k^{\gamma/n}\sqrt{h_k}\geq -
2C_{\bar{K}'}h_k^{\frac{\alpha\beta\gamma}{2n}}$$ for any $(z',0)\in \hat \pi\left(\{\bar{S}_k\cap\{x^n=P_k(x')\}\right).$ Hence \eqref{e81} follows from the above discussion easily.
Note that \eqref{e81} implies that $\0$ is almost the minimum point of $\bar{v}_k$.

Now, as in the proof of \cite[Theorem 5.3]{DF}, by standard estimates on the sections of solutions to the Monge-Amp\`ere equation it follows that
the shapes of $\{\bar{v}_k\leq 1\}$ and $\{\bar{v}_k\leq 1/2\}$ are
comparable, and in addition sections are well included into each other: hence, thanks to \eqref{eq:Sv} there exists a universal constant $L>1$ such that
$$B_{1/L\bar{K}}\subset \{\bar{v}_k\leq 1/2\}\subset B_{L\bar{K}},\quad \text{dist}\left(\{\bar{v}_k\leq 1/4\},\partial\{\bar{v}_k\leq 1/2\}\right)\geq 1/(L\bar{K}).$$
Using again \eqref{e73} we have that, if $h_1$ is sufficiently small,
$$B_{1/(2L\bar{K})}\cap \{x^n\geq P_k(x')\} \subset\{\bar{u}_k\leq 1/2\}\subset B_{2L\bar{K}}\cap \{x^n\geq P_k(x')\},$$
\begin{equation}\label{105}
\text{dist}\left(\{\bar{u}_k\leq 1/4\},\{\bar{u}_k=1/2\}\right)\geq 1/(2L\bar{K}),
\end{equation}

This allows us to apply Lemma \ref{l3} to $\bar{u}_{k+1}$ to get
\begin{equation}\label{101}
\|\bar{u}_{k+1}-\bar{v}_{k+1}\|_{L^\infty(\bar{S}_{k+1})}\leq C_{2L\bar{K}}\biggl(\underset {S_{h_{k+1}}}{\text{osc}}f+\underset {T_u(S_{h_{k+1}})}{\text{osc}}g+\delta_k^{\gamma/n}\biggr)
\leq C_{2L\bar{K}}h_k^{\frac{\alpha\beta\gamma}{2n}}.
\end{equation}
Hence, combining \eqref{e73} and \eqref{101},
\begin{eqnarray}\label{ee2}
\|v_k-v_{k+1}\|_{L^\infty(S_{h_{k+1}})}&\leq& \|v_k-u\|_{L^\infty(S_{h_k})}+\|u-v_{k+1}\|_{L^\infty(S_{h_{k+1}})}\nonumber\\
&=&h_k\|\bar{u}_k-\bar{v}_k\|_{L^\infty(\bar{S}_k)}+h_{k+1}\|\bar{u}_{k+1}-\bar{v}_{k+1}\|_{L^\infty(\bar{S}_{k+1})}\nonumber\\
&\leq& C(C_{\bar{K}'}+C_{2L\bar{K}})h_k^{1+\sigma}, 
\end{eqnarray}
where $\sigma=\frac{\alpha\beta\gamma}{2n}.$

Now, we denote $$S_{h_{k+1}}^+:=S_{h_{k+1}}\cup\Bigl\{(x',x^n): |x^n|\leq Ch_k, \,x'\in \hat \pi\left(\{S_{h_{k+1}}\cap\{x^n=P(x')\}\right)\Bigr\},$$
and we also denote $S_{h_{k+1}}^-$ as the reflection of $S_{h_{k+1}}^+$ with respect to $\{x^n=0\}.$
By scaling back \eqref{105}, we have that $\text{dist}\left(S_{h_{k+2}}, \partial (S_{h_{k+1}}^+\cup S_{h_{k+1}}^-)\right)\geq \sqrt{h_k}/(4L\bar{K}).$
Then by \eqref{e22} and \eqref{ee2} we have $\|v_k-v_{k+1}\|_{L^\infty\left(S_{h_{k+1}}^+\cup S_{h_{k+1}}^-\right)}\leq C_{\bar K',2L\bar K}h_k^{1+\sigma}$, with $C_{\bar K',2L\bar K}$ is a constant depending only on $\bar{K}'$ and $2L\bar K$, provided $\delta_1$ is small.
Hence, we can apply the classical Pogorelov and Schauder estimates to get
\begin{equation}\label{201}
\|D^2v_k-D^2v_{k+1}\|_{L^\infty(S_{h_{k+2}})}\leq C_{\bar K',2L\bar K}'h_k^\sigma,
\end{equation}
\begin{equation}\label{202}
\|D^3v_k-D^3v_{k+1}\|_{L^\infty(S_{h_{k+2}})}\leq C_{\bar K',2L\bar K}'h_k^{\sigma-1/2},
\end{equation}
where $C_{\bar K',2L\bar K}'$ is a constant depending only on $\bar{K}'$ and $2L\bar K$.
Since by the inductive assumption \eqref{e44} holds with $K=2\bar{K}$ for $h=h_j$ with $j=1,\cdots, k,$ we can apply \eqref{201} to $v_j$ to get
\begin{eqnarray}
|D^2v_{k+1}|(\0)&\leq& D^2v_1(\0)+\displaystyle{\sum_{j=1}^k}|D^2v_j(\0)-D^2v_{j+1}(\0)|\\
&\leq& M+C_{\bar K',2L\bar K}'h_1^\sigma\displaystyle{\sum_{j=0}^k}2^{-j\sigma}\\
&\leq& M+\frac{C_{\bar K',2L\bar K}'}{1-2^{-\sigma}}h_1^{\sigma}\leq M+1,
\end{eqnarray}
provided we choose $h_1$ small enough. By the definition of $\bar{K}$ it follows that also $S_{h_{k+1}}$ satisfies \eqref{e44}, concluding the proof of inductive step.
\vspace{1em}

\noindent$\bullet$ \emph{Step 2: $C^{2,\sigma'}$ estimate at $\0$.} We now prove that $u$ is $C^{2,\sigma'}$ at the origin with $\sigma':=2\sigma$, that is, there exists a sequence of paraboloids $p_k$ such that
\begin{equation}\label{203}
\underset{B_{r_{0}^k/C}}{\text{sup}}|u-p_k|\leq Cr_0^{k(2+\sigma')}
\end{equation}
for some $r_0, C>0.$

Let $v_k$ be as in the previous step, and let $p_k$ be their second order Taylor expansion at $\0$:
$$p_k(x):=v_k(\0)+\nabla v_k(\0)\cdot x+\frac{1}{2}D^2v_k(\0)x\cdot x.$$
By \eqref{e44} we have
\begin{equation}\label{204}
\|v_k-p_k\|_{L^\infty\left(B(\sqrt{h_{k+2}}/K)\right)}\leq \|v_k-p_k\|_{L^\infty(S_{h_{k+2}})}\leq C\|D^3v_k\|_{L^\infty(S_{h_{k+2}})}h_k^{3/2}.
\end{equation}
In addition, applying \eqref{202} to $v_j$ with $j=1,\cdots, k$  and recalling that $h_k=h_12^{-k}$ and $2\sigma<1$, we get
\begin{eqnarray}\label{205}
\|D^{3}v_k\|_{L^\infty(S_{h_{k+2}})}&\leq &\|D^{3}v_1\|_{L^\infty(S_{h_{3}})}+\displaystyle{\sum_{j=1}^k}\|D^{3}v_j-D^{3}v_{j+1}\|_{L^\infty(S_{h_{j+2}})}\\
&\leq& C\biggl(1+\displaystyle{\sum_{j=1}^k}h_j^{\sigma-1/2}\biggr)\leq Ch_k^{\sigma-1/2}.
\end{eqnarray}
Hence, combining \eqref{e44}, \eqref{e73}, \eqref{204}, and \eqref{205}, we get
$$\|u-p_k\|_{L^\infty\left(B(\sqrt{h_{k+2}}/K)\right)}\leq \|v_k-p_k\|_{L^\infty(S_{h_{k+2}})}+\|v_k-u\|_{L^\infty(S_{h_{k+2}})}\leq Ch_k^{1+\sigma},$$
so \eqref{203} follows with $r_0=1/\sqrt{2}$ and $\sigma'=2\sigma.$

\vspace{1em}

\noindent$\bullet$ \emph{Step 3: $C^{2,\alpha'}$ regularity near the boundary.} 
Since (recall the beginning of the proof of the theorem) the point $\0$ represented an
an arbitrary point in $B_{\rho_1/2}\cap  \{x^n=P(x')\}$, By Step 2 we know that \eqref{203}
holds at any point on $B_{\rho_1/2}\cap  \{x^n=P(x')\}$.

Set $\alpha':=\sigma'/2$.
Given $\rho>0$ let $\Omega_\rho:=\{x\in \mathcal C_1\cap B_{\rho_1/4} : d(x,\partial\Omega)\leq \rho\}.$ We want to show that if $\rho \ll \rho_1$ is sufficiently small, then $u \in C^{2,\alpha}_{\rm loc}(\Omega_\rho)$ and  $\|u\|_{C^{2,\alpha'}(\Omega_\rho)} \leq C$.
To prove this, given $x_1\in\Omega_\rho$ denote $d:={\rm dist}(x_1,\partial\mathcal C_1)$ and assume with no loss of generality that $d={\rm dist}(x_1, \0)$. Since $u$ is pointwise $C^{2,\sigma'}$ at
$\textbf{0}$ (see Step 2), after an affine transformation and change of variables similar to \eqref{e01} and \eqref{e02} we can assume that
$\|u-\frac{1}{2}|x|^2\|_{L^\infty(B_{4d})}\leq C d^{2+\sigma'}.$ Then, we perform the blow up
\[
\left\{
\begin{array}
    {l@{\quad}}
   \tilde{x}:=\frac{1}{4d}x \\
   \tilde{y}:=\frac{1}{4d}y,
\end{array}
\right.
\]
$c_1(\tilde{x},\tilde{y}):=\frac{1}{16d^2}c(4d\tilde{x}, 4d\tilde{y}),$
and $u_1(\tilde{x}):=\frac{1}{16d^2}u(4d\tilde{x}).$ In the new coordinates $B_{1/6}(\tilde{x}_1)$ lies in the interior of the domain $\widetilde{\Omega}$,
where $\tilde{x}_1:=\frac{1}{4d}x_1 $ and $\widetilde{\Omega}:=\Omega_{2\rho}/4d.$ It is immediate to check that $\|u_1(\tilde{x})-\frac{1}{2}|\tilde{x}|^2\|_{B_{1/6}(\tilde{x}_1)}\leq C d^{\alpha'}$, and $\|c(\tilde{x},\tilde{y})+\tilde{x}\cdot \tilde{y}\|_{B_{1/6}(\tilde{x}_1)\times
B_{1/6}(\tilde{y}_1)}\rightarrow 0$ as $d\rightarrow 0$, where $\tilde{y}_1\in\partial_{c_1}u_1(\tilde{x}_1).$ Hence, if $\rho$ (and therefore $d$) is sufficiently small we can apply \cite[Theorem 5.3]{DF} to deduce that
$\tilde u \in C^{2,\alpha}(B_{1/7}(\tilde x_1))$, with a universal bound.      
In particular $|D^2\tilde u(\tilde x_1)|\leq C$, thus $|D^2u(x_1)|\leq C$.
This proves that $u$ is uniformly $C^{1,1}$ inside $\Omega_\rho$,
which implies that \eqref{eq:MA T} becomes uniformly elliptic there.
Writing $\hat u(\tilde x):=\frac{1}{16d^{2+\alpha}}\left[u(x) - \frac12|x|^2\right]$,
it is easy to check that $\hat u$ solves a uniformly elliptic equation of the form
$$
G\bigl(\tilde x,\nabla \hat u(\tilde x),D^2\hat u(\tilde x)\bigr)=0
$$
where $G(x,0,0)=0$ and $\|G(\tilde x,\cdot,\cdot)\|_{C^{0,\alpha}(B_{1/6}(\tilde x_0))} \leq C$.
Hence, by standard elliptic regularity for fully-nonlinear elliptic equations we deduce that $
\|\hat u \|_{C^{2,\alpha'}(B_{1/7}(\tilde x_1))} \leq \|\hat u \|_{C^{2,\alpha}(B_{1/7}(\tilde x_1))}\leq C$.
Going back to the original coordinates, we deduce that
$$
\|u \|_{C^{2,\alpha}(B_{d/7}(x_1))} \leq Cd^{\alpha'-\alpha},\qquad
\|u \|_{C^{2,\alpha'}(B_{d/7}(x_1))} \leq C.
$$
Because of the arbitrariness of $x_1$, the first estimate proves that $u$ is of class $C^{2,\alpha}$
in the interior of $\Omega_\rho$, while the second estimate combined with 
 the fact that \eqref{203} holds at every boundary point allows one to prove by standard arguments
 (see for instance the proof of \cite[Proposition 2.4]{MS})
the $C^{2,\alpha'}$ regularity of $u$ in the whole $\Omega_\rho$.
\qed

\bibliographystyle{amsplain}

\end{document}